\documentclass[11pt]{article}

\usepackage{arxiv}
\usepackage[utf8]{inputenc}
\usepackage[T1]{fontenc}
\usepackage[pagebackref=true]{hyperref}
\usepackage{url}
\usepackage{booktabs}
\usepackage{amsfonts}
\usepackage{nicefrac}
\usepackage{microtype}
\usepackage{graphicx}
\usepackage{natbib}
\usepackage{doi}
\usepackage{dsfont}
\usepackage{amssymb,amsmath,amsfonts,amsthm}
\usepackage{verbatim}	
\usepackage{fancyvrb}
\usepackage{bbm}
\usepackage{paralist}
\usepackage{accents}
\usepackage{verbatim}
\usepackage[ruled]{algorithm2e}
\usepackage{cleveref}
\usepackage{color}
\usepackage{caption}
\usepackage{multirow}
\usepackage{stix}
\usepackage{enumitem}
\usepackage{mathtools}
\usepackage{float}
\usepackage{multicol}
\usepackage{comment}
\usepackage{import}
\usepackage{subfig}
\usepackage{tikz-cd}
\usetikzlibrary{babel}
\usepackage{accents}
\usepackage{wrapfig}
\usepackage{extarrows}
\usepackage{xstring}

\newcommand{\specialref}[1]{%
    \IfStrEq{#1}{V1}{V1}{\ref{#1}}%
}

\def\R{\mathbb{R}}
\def\N{\mathbb{N}}
\def\Z{\mathbb{Z}}
\def\Q{\mathbb{Q}}
\newcommand{\E}[1]{\mathbb{E}\left[{#1}\right]}

\def\P{\mathbb{P}}
\def\X{\mathbb{X}}

\def\A{\mathcal{A}}

\def\L{\mathbb{L}}

\def\MM{\mathbb{M}}

\def\XX{\mathcal{X}}

\def\AA{\mathbb{A}}

\newcommand{\1}{\mathds{1}}
\newcommand{\0}{\mathbf{0}}

\newcommand{\NDX}{\mathds{N}(\X)}

\renewcommand{\d}[1]{ \mathrm{d}{#1} }

\renewcommand{\|}{\mid}

\DeclareMathOperator{\diam}{diam}

\newtheorem{Th}{Theorem}[section]
\newtheorem{Prop}[Th]{Proposition}
\newtheorem{Lem}[Th]{Lemma}

\newtheorem{Def}[Th]{Definition}

\newtheoremstyle{normalstyle}
  {}
  {}
  {\normalfont}
  {}
  {\bfseries}
  {.}
  { }
  {}

\theoremstyle{normalstyle}
\newtheorem{Bsp}[Th]{Example}
\newtheorem{Alg}[Th]{Algorithm}

\hypersetup{pdftoolbar=true,
	pdfmenubar=true,
	pdfpagemode=UseOutlines,
	bookmarksnumbered=true,
	linktocpage=true,
	colorlinks=true,
	citecolor=blue,
	linkcolor=blue,
	urlcolor=blue
}

\title{Percolation in the marked stationary Random Connection Model for higher-dimensional simplicial complexes}

\author{ Dominik~Pabst \\
    Institute of Stochastics\\
    Karlsruhe Institute of Technology\\
    Institute of Theoretical Physics\\
    Friedrich Alexander University Erlangen-Nuremberg
}

\renewcommand{\headeright}{}
\renewcommand{\undertitle}{}
\renewcommand{\shorttitle}{Percolation in the RCM for higher-dimensional simplicial complexes}

\hypersetup{
pdftitle={Percolation in the marked stationary Random Connection Model for higher-dimensional simplicial complexes},
pdfauthor={Dominik~Pabst},
pdfkeywords={Random Connection Model, Percolation, Simplicial complexes, Sharp phase transition},
}

\begin{document}

\maketitle

\begin{abstract}
   We introduce a novel percolation model that generalizes the classical Random Connection Model (RCM) to a random simplicial complex, allowing for a more refined understanding of connectivity and emergence of large-scale structures in random topological spaces.
   Regarding percolation with respect to the notion of up-connectivity, we establish the existence of a sharp phase transition for the appearance of a giant component, akin to the well-known threshold behavior in random graphs.
   This sharp phase transition is, in its generality, new even for the classical RCM as a random graph.
   As special cases, we obtain sharp phase transitions for the Vietoris-Rips complex, the \v Cech complex, and the Boolean model, allowing us to identify which properties of these well-known percolation models are actually required.
\end{abstract}

\begin{small}
\end{small}

\begin{small}
\mscs{60K35 \and 60D05 \and 60G55 \and 05C80}
\end{small}

\section{Introduction}

Motivated by physical problems, percolation has become a topic of significant mathematical interest, with many open questions remaining. 
Typical questions in percolation theory concern the existence and uniqueness of infinite components in random graphs. 
In the classical discrete percolation model, one considers the random subgraph of the lattice $\mathbb{Z}^d$ for $d \geq 2$, where each edge is retained independently with some fixed probability $p \in [0,1]$. 
It is well-known that there exists a critical probability $p_c$ such that there is almost surely no infinite component for $p < p_c$, while there is almost surely exactly one unique infinite component for $p > p_c$. 
The behavior of percolation models at or near their critical threshold is a central theme in percolation theory.

In this paper, we consider a continuous model, known as the \textit{Random Connection Model} (RCM), which has been investigated with respect to a wide variety of mathematical questions (see, for example, \cite{Can,Last.RCM}), particularly in the context of deriving qualitative and quantitative central limit theorems for fundamental graph characteristics. 
Originally, however, it was introduced as a percolation model in \cite{Penrose.RCM}. 
In this context, the vertices of the random graph are the points of a Poisson process, say on $\mathbb{R}^d$ for the moment, with intensity measure $\beta\lambda_d$ for some $\beta > 0$. 
For each pair of vertices $x,y$, an edge between them is independently included in the graph with probability $\varphi(x,y)$, where $\varphi:\mathbb{R}^d \times \mathbb{R}^d \rightarrow [0,1]$ is a measurable, symmetric, and translation-invariant function. 
It was proven in \cite{Penrose.RCM} that under a weak condition there exists a critical intensity $0 < \beta_c < \infty$, such that there is almost surely an infinite component for $\beta > \beta_c$ and almost surely no infinite component for $\beta < \beta_c$. 
One method for analyzing the phase transition at the critical point is to introduce alternative definitions of critical intensity and compare them. 
A particularly useful technique involves adding a point at the origin to the Poisson process (due to translation invariance, the specific location of this additional point is irrelevant) and incorporating it into the graph. 
This allows us to define a new critical intensity as the supremum of all $\beta \in [0,\infty)$ for which the expected cluster size of the added point remains finite. 
Theorem 6.2 in \cite{Meester} establishes that this critical intensity, denoted by $\beta_T$, coincides with the previously introduced critical intensity. 
However, this equality does not necessarily hold in the more general framework introduced below.

In this work, we consider the RCM, where the points in $\mathbb{R}^d$ are endowed with an additional component, which can be thought of as additional information, and extend this model to a random simplicial complex. 
This model has already been introduced and explored in \cite{Pabst.Betti,Pabst.Euler}, where central limit theorems for typical quantities of simplicial complexes were established. 
In \cite{Pabst.Euler}, the model was examined in an even more general framework, allowing vertices to be drawn from an arbitrary Borel space.
In that work, central limit theorems for the Euler characteristic were established across various asymptotic scenarios with rates of convergence.
In contrast, \cite{Pabst.Betti} presents central limit theorems for a class of functionals, including the Betti numbers, within the stationary marked case, which will also be examined in this study.
For this purpose, let $\Phi$ be a Poisson process on $\mathbb{R}^d \times \AA$, where $(\AA,\mathcal{T})$ is an arbitrary Borel space, called the mark space, equipped with a probability measure $\Theta$, with intensity measure $\beta\lambda_d \otimes \Theta$. 
We fix $\alpha \in \mathbb{N}$ and for each $j \in \{1,\dots,\alpha\}$ define a measurable, symmetric, and translation-invariant function $\varphi_j:(\mathbb{R}^d \times \AA)^{j+1} \rightarrow [0,1]$, where the latter means that the value of the function does not change when the Euclidean components of all points are translated by the same vector (see \eqref{translation_invariant}). 
We call the functions $\varphi_1,\dots,\varphi_\alpha$ the connection functions of the model.

A simplicial complex is a family of nonempty, finite subsets of a vertex set, which is closed under taking subsets. 
Elements of a simplicial complex are called simplices and represent interactions or connections between the vertices. 
A simplex containing $j + 1$ vertices is called a $j$-simplex with $j$ as its dimension. 
The dimension of a simplicial complex is the supremum over the dimensions of its simplices, so a graph can be viewed as a (at most) one-dimensional simplicial complex. 
We define a random simplicial complex $\Delta$ in several steps:
\begin{itemize}
\item[(0)] The vertex set of $\Delta$ is the set of points of $\Phi$.
\item[(1)] For each pair of points $x,y \in \Phi$, add the edge $\{x,y\}$ independently with probability $\varphi_1(x,y)$ to the complex $\Delta$.
\item[(2)] For each triple $x,y,z \in \Phi$, whose edges have all been added to the complex, add the triangle $\{x,y,z\}$ independently with probability $\varphi_2(x,y,z)$ to the complex $\Delta$.
\item[$\vdots$]
\item[($\alpha$)] For $\alpha + 1$ points $x_{i_0},\dots,x_{i_\alpha} \in \Phi$, whose subsimplices have all been added to the complex, add the simplex $\{x_{i_0},\dots,x_{i_\alpha}\}$ independently with probability $\varphi_\alpha(x_{i_0},\dots,x_{i_\alpha})$ to $\Delta$.     
\end{itemize}
In the case $\alpha = 1$, we have the RCM as a random graph that is studied in the literature. 
Like the RCM, this model can also be defined and investigated on even more general spaces. 
For example, in \cite{Chebunin}, the RCM is considered on complete separable metric spaces, and Theorem 6.1 there shows that under two conditions there is almost surely at most one infinite cluster. 
However, the methods we use in this paper crucially require properties of Euclidean space. Nevertheless, we do not need to restrict ourselves completely to Euclidean space, but instead can consider the previously introduced situation, which we refer to as the \textit{marked stationary} RCM. 
The marked stationary RCM (for a random graph) has been studied in \cite{Caicedo,Dickson}.
Note that we obtain the unmarked case (the RCM on $\mathbb{R}^d$) as a special case by choosing a single-point mark space or by connection functions that depend only on the Euclidean component. 
Another important case is given by choosing $\AA := \{ K \subset \mathbb{R}^d \mid \emptyset \neq K \text{ is convex and compact}\}$, with the connection functions
\begin{align} \label{connectionfct_BoolM}
    \varphi_j\big((x_0,K_0),\dots,(x_j,K_j)\big) \,=\, \1\Big\{ \bigcap_{i=0}^j (x_i + K_i) \neq \emptyset \Big\}, \qquad j\in\{1,\dots,\alpha\}.
\end{align}
This subcase is directly related to the Boolean model, one of the flagship models of stochastic geometry. 
See Section \ref{Sec:ExampleModels} for details. 
In the Boolean model, the equality $\beta_c = \beta_T$ is not generally true (see Chapter 3 in \cite{Meester}), but Proposition 2.2 in \cite{Dickson} shows that equality holds if the mark distribution $\Theta$ is concentrated on balls with radius not larger than a certain constant. 
All these results pertain to the percolation of graphs.

In recent years, \textit{simplicial complexes} have increasingly emerged as a generalization of graphs in connection with percolation problems. 
First, the question arises as to how the concept of percolation can be transferred to simplicial complexes. 
Certainly, one could speak of percolation in simplicial complexes just as in graphs if there exists an infinite connected component. 
However, this usage of the term would not be interesting, as a simplicial complex would percolate in this case if and only if its restriction to simplices of dimension at most 1 percolates.
This means that simplices with dimensions greater than 1 would not play any role, and therefore, effectively, one would only consider a graph again. 
Thus, we first need to clarify what we mean by percolation in a simplicial complex. 
In the literature, there are several concepts of percolation for simplicial complexes. 
For example, in \cite{Hirsch}, face percolation, cycle percolation, and *-percolation are considered for the Vietoris-Rips complex. 
Face percolation is simply another term for percolation with respect to the so-called down connectivity, which, along with up connectivity, is addressed in \cite{Iyer} for the Vietoris-Rips complex and the \v Cech complex. 
\begin{figure}[t!]
    \begin{minipage}{.5\linewidth}
    \centering
        \def\svgwidth{1.0\textwidth}
        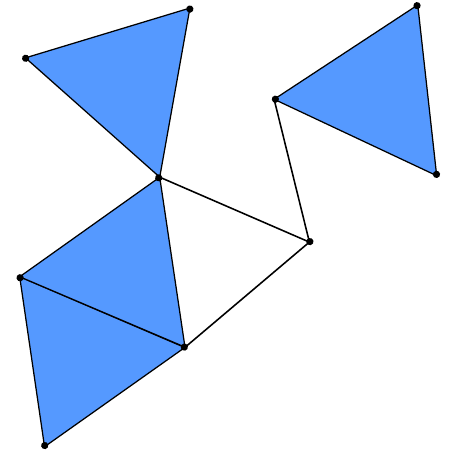
    \end{minipage}
    \begin{minipage}{.5\linewidth}
    \centering
        \def\svgwidth{1.0\textwidth}
        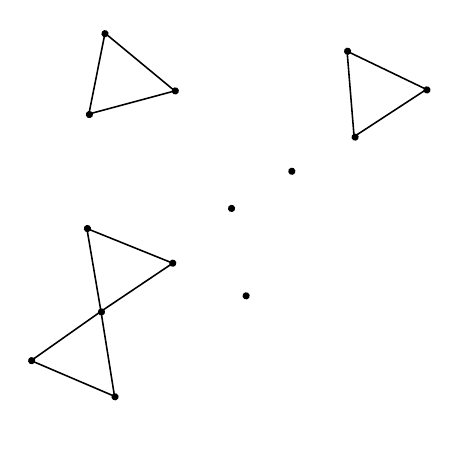
    \end{minipage}
    \caption{A two-dimensional simplicial complex $K$ (left) and its 1-graph $G_1(K)$ (right)}
    \label{fig:UpGraph}
\end{figure}
In this paper, we focus on percolation with respect to up connectivity, while also demonstrating that our results hold for percolation regarding down connectivity.
The idea behind these two concepts of connectivity is to construct a graph from a simplicial complex and then to speak of the percolation of the simplicial complex when the constructed graph percolates in the classical sense, meaning it has an infinite connected component. 
To do this, we fix some $q \in \mathbb{N}_0$ and define the vertex set of the $q$-graph $G_q(K)$ of a simplicial complex $K$ as the set of $q$-simplices of $K$. 
There is an edge between two $q$-simplices $\sigma,\rho$ of $K$ if $K$ contains a $(q+1)$-simplex $\pi$ with $\sigma,\rho \subset \pi$. 
Evidently, the $q$-graph of a simplicial complex $K$ only depends on its restriction to simplices of dimension at most $q+1$, which is referred to as the $(q+1)$-skeleton of $K$ and is again a simplicial complex. 
For $q=0$, the $q$-graph of $K$ is simply its 1-skeleton, disregarding the fact that $0$-simplices are not vertices, but rather one-element subsets of the vertex set. 
Therefore, the $0$-graph of a simplicial complex is connected if and only if the simplicial complex itself is connected in the classical sense. 
In our situation, only the case $q < \alpha$ is of interest, since for $q > \alpha$ we would obtain the empty graph and for $q = \alpha$ we get a graph without edges. 
Figure \ref{fig:UpGraph} illustrates the 1-graph of a two-dimensional simplicial complex $K$. 
Here, one can immediately see that the $q$-graph of a connected simplicial complex does not necessarily have to be connected. 
The reverse implication is also not true. 
Figure \ref{fig:UpGraph} further demonstrates how every $(q+1)$-simplex in $K$ induces a complete subgraph on $q+2$ vertices in $G_q(K)$. 

The main goal of this work is to prove a \textit{sharp phase transition} for the percolation function defined later, by which we mean a result similar to Theorem 1 in \cite{Hirsch} or Theorem 8.2 in \cite{Last.OSSS}. 
The sharp phase transition in \cite{Hirsch} considers percolation with respect to down connectivity of the Vietoris-Rips complex, while the result in \cite{Last.OSSS} refers to percolation of the Boolean model in a sense that is equivalent to percolation with respect to up connectivity in the model introduced with the connection functions from \eqref{connectionfct_BoolM}. 
We obtain both the Vietoris-Rips complex and the \v Cech complex as special cases of the unmarked case by choosing suitable connection functions (see Section \ref{Sec:ExampleModels}). 
More precisely, this gives us only the $\alpha$-skeletons of the two complexes. 
However, since the $q$-graph of a simplicial complex depends only on its $(q+1)$-skeleton, this does not impose any restriction.
The main tools for proving the sharp phase transition are the discrete OSSS inequality and the Margulis-Russo formula for Poisson processes.
These two results play a crucial role in numerous proofs regarding the properties of percolation models (see, for example, \cite{Duminil-Copin,Hirsch}).
Throughout this paper, we will identify suitable conditions for the connection functions, specifically the properties \ref{V1} and \ref{V2}, that facilitate the application of the proof methods.
Heuristically speaking, \ref{V1} states that vertices far apart are not connected by $(q+1)$-simplices.
On the other hand, \ref{V2} asserts that closely situated vertices form a $(q+1)$-simplex with at least a given probability, provided their marks belong to a specific set.
In \cite{Pabst.Thesis}, additional percolation formulas for graphs are extended to simplicial complexes, specifically the FKG inequality and the BK inequality.
In this paper, results and methods derived from the PhD thesis \cite{Pabst.Thesis} are presented.

In Section \ref{Sec:Definition_Model}, we offer an exact definition of the model.
The specific construction of the simplicial complex $\Delta$ introduced there is essential for employing the methods to derive a sharp phase transition.
Before we begin deriving the sharp phase transition, we will investigate the critical intensities for percolation in the model in Section \ref{Sec:CriticalInt}.
The main result of this paper is Theorem \ref{Th:scharferPhasenübergang}, which is established in Section \ref{Sec:SharpPhaseTransition}.
In conclusion, Section \ref{Sec:ExampleModels} presents several example models that satisfy the requirements of Theorem \ref{Th:scharferPhasenübergang}.

\section{Preliminaries} \label{Sec:Pre}

In this section we give precise definitions of the objects from the introduction and state some important properties and notions connected to those objects, which are important for this paper.
Throughout this work $(\AA,\mathcal{T})$ will always be a Borel space, that is a measurable space such that a bijective measurable map $\phi:\AA\rightarrow U$ into a measurable subset $U$ of $[0,1]$ exists with measurable inverse.
Additionally we fix some probability measure $\Theta$ on $\AA$.

\subsection{Poisson processes}

Throughout this paper, we will only consider Poisson processes on the product of $\R^d$ with some Borel space, which is again a Borel space.
Therefore, we introduce Poisson processes in the context of Borel spaces, although the theory of Poisson processes allows for a much more general setting.
For a detailed introduction to the theory, we refer to \cite{Last.RCM}.
Let $(\X,\XX)$ be a Borel space and $\NDX$ be the set of all measures on $\X$ that can be expressed as a countable sum of counting measures (measures with values in $\N_0$).
We equip $\NDX$ with the smallest $\sigma$-field such that the mappings $\NDX\rightarrow\R$, $\eta\mapsto\eta(B)$ are measurable for all measurable sets $B\subseteq\X$.
A point process on $\X$ is a random element in $\NDX$, and a Poisson process $\xi$ with intensity measure $\lambda$ is a point process satisfying the following two properties:
\begin{enumerate}
    \item[(a)] For every measurable set $B\subseteq\X$, the random variable $\xi(B)$ follows a Poisson distribution with parameter $\lambda(B)$.
    \item[(b)] For pairwise disjoint measurable sets $B_1,\dots,B_m\subseteq\X$, $m\in\N$, the random variables $\xi(B_1),\dots,\xi(B_m)$ are stochastically independent.
\end{enumerate}
Now let $\xi$ be a Poisson process on $\R^d$ with intensity measure $\beta\lambda_d$ for some $\beta>0$, where $\lambda_d$ denotes the $d$-dimensional Lebesgue measure on $\R^d$.
As a consequence of Corollary 6.5 in \cite{Last.Lectures}, we can find random elements $X_1,X_2,\dots$ in $\R^d$, such that
\begin{align*}
    \xi \,=\, \sum_{i=1}^\infty \delta_{X_i}
\end{align*}
holds almost surely.
As a consequence of Proposition 6.9 in \cite{Last.Lectures}, the points $X_1,X_2,\dots$ are almost surely pairwise distinct.
Now let $A_1,A_2,\dots$ be i.i.d. random elements in $\AA$ with distribution $\Q$.
Then the point process process
\begin{align*}
    \Phi \,=\, \sum_{i=1}^\infty \delta_{(X_i,A_i)}
\end{align*}
on $\R^d\times\AA$ is called an independent $\Q$-marking of $\xi$.
Theorem 5.6 in \cite{Last.Lectures} shows that $\Phi$ is a Poisson process on $\R^d\times\AA$ with intensity measure $\beta\lambda_d\otimes\Theta$.
Since the distribution of a Poisson process only depends on its intensity measure (see Proposition 3.2 in \cite{Last.Lectures}), every Poisson process on $\R^d\times\AA$ with intensity measure $\beta\lambda_d\otimes\Theta$ has the same distribution as $\Phi$.

\subsection{Percolation of simplicial complexes}

Given an at most countable vertex set $V$, we call a family $K$ of nonempty, finite subsets of $V$ that is closed under taking subsets a simplicial complex with vertex set $V$.
For convenience, we always assume that all one-element subsets of $V$ are contained in $K$, so that the vertex set of a simplicial complex is uniquely determined.
An element $\sigma\in K$ containing $j+1$ vertices is called a $j$-simplex with $\dim(\sigma)=j$.
A simplex consisting of the vertices $v_0,\dots,v_j$ is denoted by $[v_0,\dots,v_j]$.
The dimension of $K$ is $\dim(K):=\sup\{ \dim(\sigma)\mid\sigma\in K\}$, which can be infinite; however, in this work, we only consider finite-dimensional simplicial complexes.
The restriction of a simplicial complex $K$ to simplices of dimension not bigger than $j\in\N_0$ is denoted by
\begin{align*}
   S_j(K) \,:=\, \{ \sigma\in K \,\|\, \dim(\sigma)\leq j \}
\end{align*}
and is referred to as the $j$-skeleton of $K$.
To study percolation of simplicial complexes, we utilize the notion of up connectivity, as discussed in \cite{Iyer} alongside down connectivity.
In both cases, a graph is constructed from a simplicial complex, which can then be examined for connectivity in the classical sense.
Section 1.1 of \cite{Iyer} discusses applications of both concepts.
In this paper, we focus on percolation with respect to up connectivity and define the $q$-graph of a simplicial complex.
For a simplicial complex $K$ and $q\in\N_0$, we define the $q$-graph $G_q(K)$ of $K$ on the vertex set $F_q(K):=\{\sigma\in K\mid \dim(\sigma)=q\}$ by the rule
\begin{align} \label{Rule_UpGraph}
    [\sigma,\rho]\in G_q(K) \enskip :\Longleftrightarrow\enskip \exists\pi\in F_{q+1}(K) \text{ with } \sigma,\rho\subset\pi
\end{align}
for $\sigma,\rho\in F_q(K)$.
Observe that two $q$-simplices $\sigma,\rho$ of $K$ can only be connected in $G_q(K)$ by an edge if $|\sigma \cap \rho| = q$.
Moreover, if $[\sigma,\rho] \in G_q(K)$, the relevant simplex $\pi$ from (\ref{Rule_UpGraph}) is given by $\sigma \cup \rho$. We say that $K$ $q$-percolates if $G_q(K)$ has an infinite component.
It is easy to show that if a simplicial complex $q$-percolates, it also $r$-percolates for all $r<q$.
As already hinted, this notion of percolation is closely related to percolation with respect to down connectivity.
To explain the duality between these two notions, we define the graph $\tilde{G}_{q+1}(K)$ over the vertex set $F_{q+1}(K)$ by
\begin{align*}
    [\pi,\tau]\in \tilde{G}_{q+1}(K) \enskip :\Longleftrightarrow\enskip \exists\sigma\in F_{q}(K) \text{ with } \sigma\subset\pi,\tau
\end{align*}
for $\pi,\tau\in F_{q+1}(K)$.
Here, a necessary and sufficient condition for $[\pi,\tau]\in \tilde{G}_{q+1}(K)$ is $|\pi\cap\tau|=q+1$, because we can always choose $\sigma=\pi\cap\tau$.
The duality of those two concepts arises from the fact that for a locally finite simplicial complex $K$ (every vertex is only part of finitely many simplices), there is an infinite component in $\tilde{G}_{q+1}(K)$ if and only if there is an infinite component in $G_q(K)$.
Keeping that in mind, we will focus on percolation with respect to $G_q(K)$ in this paper.

\section{Definition of the model} \label{Sec:Definition_Model}

To define the random simplicial complex $\Delta$ we fix some $\alpha\in\N$, which will be the maximal dimension of $\Delta$, and for each $j\in\{1,\dots,\alpha\}$ a measurable, symmetric and translation-invariant function $\varphi_j:(\R^d\times\AA)^{j+1}\rightarrow [0,1]$.
Here by translation-invariant we mean that
\begin{align} \label{translation_invariant}
    \varphi_j\big((x_0+t,a_0),\dots,(x_j+t,a_j)\big) \,&=\, \varphi_j\big((x_0,a_0),\dots,(x_j,a_j)\big)
\end{align}
holds for all $x_0,\dots,x_j,t\in\R^d$ and $a_0,\dots,a_j\in\AA$.
To decide which simplices are contained in $\Delta$, we will consider an even bigger space than $\R^d\times\AA$.
For $j\in\N$ let
\begin{align*}
M^{(j)} \,:=\, [0,1]^{\N_0^j} \,=\, \big\{ (a_{z})_{z\in\N_0^j} \,\|\, a_{z}\in [0,1] \text{ for all } z\in\N_0^j \big\}
\end{align*}
be the space of $j$-times indexed sequences with members in $[0,1]$ equipped with the product $\sigma$-field of the Borel $\sigma$-fields on $[0,1]$.
We consider the space $\MM := M^{(2)} \times \dots \times M^{(2\alpha)}$ together with the product $\sigma$-field and the probability distribution $\Q:=\otimes_{j=1}^\alpha\otimes_{z\in\N^{2j}}\mathcal{U}([0,1])$, where $\mathcal{U}([0,1])$ is the uniform distribution on $[0,1]$.
For an element $u\in\MM$, $j\in\{1,\dots,\alpha\}$ and $m_1,l_1,\dots,m_j,l_j\in\N_0$ we denote by $u_{m_1,l_1,\dots,m_j,l_j}$ the component of $u$ associated to $z=(m_1,l_1,\dots,m_j,l_j)\in\N_0^{2j}$.
Since a countable product of Borel spaces is again a Borel space, the situation when considering $\R^d\times\AA\times\MM$ has not become more general.
For the rest of the paper let $\Psi$ be a Poisson process on $\R^d\times\AA\times\MM$ with intensity measure $\beta\lambda_d\otimes\Theta\otimes\Q$ and $\Phi$ be its projection on $\R^d\times\AA$, which is a Poisson process with intensity measure $\beta\lambda_d\otimes\Theta$.
The points of $\Phi$ will be the vertices of $\Delta$ and the additional components in $\MM$ will be used to decide which simplices are part of $\Delta$.
As a result the whole randomness of the model is bundled in one Poisson process allowing us to make use of the powerful theory of Poisson processes. \\
To provide an explicit rule for constructing $\Delta$, we decompose the Euclidean space $\R^d$ into cubes of edge length $2t$ for some $t>0$.
The distribution of $\Delta$ is independent of the choice of $t$, which is why any arbitrary value can be used here in principle.
However, at a later stage, it makes sense to choose a specific value depending on the connection functions.
We fix an enumeration $z_0,z_1,\dots$ of $\Z^d$ with $z_0=\0$ (origin) and define a partition $\mathcal{Q}^t=(Q_i^t)_{i\in\N_0}$ of $\R^d$ by $Q_i^t:=[-t,t)^d+2t z_i$.
Given a realisation of $\Phi$ with $(x,a)\in\Phi$, we assign a tuple $(m,l)\in\N_0^2$ to $(x,a)$, which can be thought of the coordinates of the point $(x,a)$ inside of $\Phi$.
Here $m\in\N_0$ is the unique number with $x\in Q_m^t$ and $l$ is the number of points $(y,b)$ of $\Phi$ with $y\in Q_m^t$ and $y<x$, where $<$ is the lexicographical order on $\R^d$.
Accordingly, a point of $\Phi$ with coordinates $(m,l)$ is the $(l+1)$-smallest point (with respect to the lexicographical order of the Euclidean components) of $\Phi$ inside the cube $Q_m^t$.
In this way, the coordinates of all points of $\Phi$ are $\P$-almost surely uniquely determined.
For $j+1$ points $(x_0,a_0,u^{(0)}),\dots,(x_j,a_j,u^{(j)})\in\Psi$, $1\leq j\leq \alpha$, we consider the possible $j$-simplex $\sigma:=\{(x_0,a_0),\dots,(x_j,a_j)\}$ of $\Delta$ and define $u(\sigma):=u^{(j)}_{m_0,l_0,\dots,m_{j-1},l_{j-1}}$, where $(m_k,l_k)$ are the coordinates of $(x_k,a_k)$ inside $\Phi$.
Now we define the complex $\Delta$ by the decision rule
\begin{align*}
    \sigma\in\Delta \enskip :\Longleftrightarrow \enskip u(\rho)\leq \varphi_{|\rho|-1}(x_\rho) \enskip \text{for all } \rho\subseteq\sigma \text{ with } |\rho|\geq 2,
\end{align*}
where $\varphi_{|\rho|-1}(x_\rho)$ is the value of $\varphi_{|\rho|-1}$, when inserting the elements of $\rho$.
In this manner the simplicial complex $\Delta$ is defined for almost all realisations of $\Psi$.
A similar way of assigning marks to the edges in the RCM was used in \cite{Last.RCM}.
We denote the restriction of the simplicial complex $\Delta$ to simplices with vertices in $W\times\AA$ for some $W\subset \mathbb{R}^d$ by $\Delta_W$.
Similarly, we write $\Phi_W$ and $\Psi_W$ for the restrictions of $\Phi$ and $\Psi$ to $W\times\AA$ and $W\times\AA\times\MM$, respectively.
A crucial property of this construction for the methods employed in this work is the following.
Observing that the coordinates of a point of $\Phi$ in $Q_i^t\times\AA$ depend only on $\Phi_{Q_i^t}$, we conclude that for a union $W$ of cubes from the decomposition $\mathcal{Q}^t$ the simplicial complex $\Delta_W$ is determined by $\Psi_W$.
Finally, we would like to emphasize that it is possible to formally introduce a space of simplicial complexes.
However, since this is not necessary here, we will refrain from doing so.

\section{Critical Intensities} \label{Sec:CriticalInt}

In this section, we focus on critical intensities for the percolation of the random simplicial complex $\Delta$ defined in Section \ref{Sec:Definition_Model}.
In the context of percolation models, various definitions for critical intensities are frequently discussed and compared (see, for instance, Proposition 2.1 in \cite{Dickson} or Theorem 2 in \cite{Hirsch}), but we will adopt a classical definition in this work.
We assume $d \geq 2$ in this context to guarantee the finiteness of the critical intensities defined later.
For the remainder of this chapter, we use the following simplified notations.
For a simplex $\sigma$ with vertices in $\R^d \times \AA$ and $x \in \R^d$, we write $x \in \sigma$ if there exists an $a \in \AA$ with $(x, a) \in \sigma$.
We define the diameter $\diam(\sigma)$ of a simplex $\sigma = [(x_0, a_0), \dots, (x_s, a_s)]$ as $\diam(\sigma) := \diam(\{x_0, \dots, x_s\})$, where $\diam(B)$ refers to the diameter of the set $B \subseteq \R^d$ according to the Euclidean metric.
Finally, let $V \sim \Theta$ and $U \sim \Q$ always be independent marks (in particular, independent of $\Psi$).
To define the critical intensities, we incorporate the dependence of the probability measure on the intensity $\beta$ into the notation and denote it by $\P_\beta$.

\begin{Def}
Fix $q \in\{0, \dots, \alpha-1\}$, and denote by $\Delta^\0$ the simplicial complex constructed from $\Psi+\delta_{(\0,V,U)}$ in exactly the same way as $\Delta$ is constructed from $\Psi$.
\begin{enumerate}
    \item Let $C_q^\infty$ denote the event that there exists a $q$-simplex $\sigma$ in $\Delta^\0$ with $\0 \in \sigma$ and an infinite connected component in $G_q(\Delta^\0)$.  
    \item Define  
    \begin{align*}
        \beta_{\mathrm{c}}^{(q)} \,&:=\, \sup \big\{\beta\in [0,\infty) \,\|\, \P_\beta( C_q^\infty )=0 \big\}.
    \end{align*}
\end{enumerate}
\end{Def}

For the remainder of this paper, we consider connection functions that satisfy two properties, for the definition of which we introduce additional notation.
For $j \in\{1, \dots, \alpha\}$, we define the function $\kappa_j:(\R^d\times\AA)^{j+1}\rightarrow [0,1]$ by
\begin{align} \label{Def:kappa_j}
\kappa_j\big((x_0,a_0),\dots,(x_{j},a_j)\big) \,:=\, \prod_{\emptyset\neq I\subseteq \{0,\dots,j\}} \varphi_{|I|-1}(x_I),
\end{align}
where $\varphi_0\equiv 1$  and $\varphi_{|I|-1}(x_I)$ denotes the value of the function $\varphi_{|I|-1}$ when the points $(x_i,a_i)$, $i\in I$, are inserted.
The value of the function in \eqref{Def:kappa_j} is the probability that the simplex $[(x_0, a_0), \dots, (x_j, a_j)]$ is part of the complex $\Delta$, given that the points $(x_0, a_0), \dots, (x_j, a_j)$ are points of $\Phi$ (i.e., vertices of $\Delta$).
In the special cases $j=1$ and $j=2$, we obtain
\begin{align*}
\kappa_1 \,=\, \varphi_1,\qquad \kappa_2(x,y,z) \,=\, \varphi_2(x,y,z)\cdot\varphi_1(x,y)\cdot\varphi_1(x,z)\cdot\varphi_1(y,z)
\end{align*}
for $x,y,z\in\R^d\times\AA$.
When studying $q$-percolation of $\Delta$ for $q\in\{0, \dots, \alpha - 1\}$, we assume that the connection functions satisfy the following two properties.

\begin{enumerate}[label=(V\arabic*)]
    \item\label{V1} There exist $\delta, \varepsilon > 0$ and a measurable set $A \subseteq \AA$ with $\Theta(A) > 0$ such that
    \begin{align*}
        \kappa_{q+1}\big( (x_0,a_0),\dots,(x_{q+1},a_{q+1}) \big) \,\geq\, \varepsilon \qquad &\text{for all } a_0,\dots,a_{q+1} \in A \text{ and } x_0,\dots,x_{q+1} \in \R^d \\
        &\text{with } \diam\big(\{x_0,\dots,x_{q+1}\}\big) \leq \delta.
    \end{align*}
    \item\label{V2} There exists a $D > 0$ such that
    \begin{align*}
        \kappa_{q+1}\big( (x_0,a_0),\dots,(x_{q+1},a_{q+1}) \big) \,=\, 0 \hspace{19pt} &\text{for $\Theta^{q+2}$-almost all } (a_0,\dots,a_{q+1}) \in \AA^{q+2} \text{ and } \\
        &x_0,\dots,x_{q+1} \in \R^d \text{ with } \diam\big(\{x_0,\dots,x_{q+1}\}\big) > D.
    \end{align*}
\end{enumerate}

\noindent Condition \ref{V2} can be interpreted as a boundedness condition for the connection functions and ensures that $\P$-almost surely in $G_q(\Delta^\0)$, only simplices $\sigma, \rho \in F_q(\Delta^\0)$ with $\diam(\sigma \cup \rho) \leq D$ are connected by an edge.
Loosely speaking, \ref{V2} states that, $\P$-almost surely, points that are far away from each other cannot form a $(q+1)$-simplex in $\Delta$.  
On the other hand, \ref{V1} ensures that, given $\Phi$, points that are close to each other and have marks in $A$ will, with at least probability $\varepsilon > 0$, form a $(q+1)$-simplex in $\Delta$.
Moreover, \ref{V1} provides an irreducible submodel (by restricting to vertices with marks in $A$) in the sense of (5.1) in \cite{Chebunin}.  
Overall, we obtain the functional bounds
\begin{align*}
    \varepsilon\; \1\big\{ \diam\big(\{x_0,\dots,x_{q+1}\}\big)\leq\delta, a_0,\dots,a_{q+1}\in A \big\} \,&\leq\, \kappa_{q+1}\big( (x_0,a_0),\dots,(x_{q+1},a_{q+1}) \big) \\
    &\leq\, \1\big\{ \diam\big(\{x_0,\dots,x_{q+1}\}\big)\leq D \big\},
\end{align*}
which hold $(\lambda_d\otimes\Theta)^{q+2}$-almost everywhere.
It is evident that $D \geq \delta$ must hold.  
Property \ref{V1} is transferred by the definition of the functions $\kappa_1, \dots, \kappa_{q+1}$ to the functions $\kappa_1, \dots, \kappa_q$.  
Conversely, \ref{V2} is also satisfied for all $\kappa_j$ with $j > q+1$ (although these functions are irrelevant for $q$-percolation).
Additionally, we would like to point out that \ref{V1} implies
\begin{align} \label{Integrabilität:marked}
\beta\int_{\R^d} \int_\AA \int_\AA \varphi_1\big( (\0,a),(y,b) \big) \;\Theta (\d a) \; \Theta (\d b) \; \d y \,>\, 0.
\end{align}
The integral in \eqref{Integrabilität:marked} corresponds to the expected edge degree of the added point in $\Delta^\0$.
The finiteness of the integral is generally not given initially.  
However, if we replace the edge function by the edge function
\begin{align} \label{Modif:Kantenfkt}
\tilde{\varphi}_1\big( (x,a),(y,b) \big) \,:=\, \varphi_1\big( (x,a),(y,b) \big) \; \1\big\{ \Vert x-y\Vert \leq D \big\}
\end{align}
the probability that $G_q(\Delta)$ has an infinitely large connected component remains unchanged, since by \ref{V2}, simplices with diameter greater than $D$ almost surely form isolated vertices in $G_q(\Delta)$. 
Thus, if we are only interested in the percolation of $G_q(\Delta)$, we can always restrict to the subcomplex of $\Delta$ consisting of all simplices of $\Delta$ with diameter no greater than $D$.
For the edge function from (\ref{Modif:Kantenfkt}), which corresponds to this subcomplex, the integral from (\ref{Integrabilität:marked}) is finite.
The finiteness of the integral in \eqref{Integrabilität:marked} corresponds to condition (4.2) in \cite{Chebunin} and is much weaker than the integrability condition (D.1) from \cite{Caicedo}.

\begin{Th} \label{Th:kritischeIntensität}
Assume that the connection functions satisfy \ref{V1} for $q=\alpha-1$ and \ref{V2} for $q=0$.
Then
\begin{align*}
0\,<\,\beta_c^{(0)}\,\leq\, \beta_c^{(1)}\,\leq\,\dots\,\leq\, \beta_c^{(\alpha-1)} \,<\,\infty.
\end{align*}
\end{Th}

\begin{proof}
Note that the assumptions imply, that \ref{V1} and \ref{V2} hold for all $q\in\{0,\dots,\alpha-1\}$.
The positivity of $\beta_c^{(0)}$ is demonstrated in \cite{Caicedo} and \cite{Dickson} (where it is denoted by $\lambda_c$).  
More precisely, Lemma 2.2 in \cite{Caicedo} establishes the positivity of a differently defined critical intensity, denoted there by $\lambda_O$, under an integration condition, which follows directly from \ref{V2}.  
Note that although a Polish space is chosen as the mark space in \cite{Caicedo}, this is not necessary for this result.
For the exact definition of $\lambda_O$, we refer to \cite{Caicedo,Dickson}.  
Proposition 2.1 in \cite{Dickson} shows that $\lambda_c\geq\lambda_O$, which implies $\beta_c^{(0)}=\lambda_c>0$.  
Although \cite{Dickson} assumes a reflection-invariant edge function, we can circumvent this assumption without verifying the proof of Proposition 2.1 in \cite{Dickson} in detail.  
Due to \ref{V2}, we can bound the edge function $\varphi_1$ from above by the reflection-invariant function  
\begin{align*}
\hat{\varphi}_1\big( (x,a),(y,b) \big) \,:=\, \1\big\{ \Vert x-y\Vert\leq D \}
\end{align*}
and, using the preceding argument, conclude the positivity of the critical intensity with respect to $\hat{\varphi}_1$.  
Since $\varphi_1\leq\hat{\varphi}_1$, the critical intensity corresponding to $\varphi_1$ is at least as large as that corresponding to $\hat{\varphi}_1$ and is therefore itself positive.  
The estimate $\beta_c^{(q-1)}\leq \beta_c^{(q)}$ follows from the fact that $q$-percolation implies $r$-percolation for all $r<q$. \\
It remains to show that $\beta_c^{(q)}<\infty$ for $q = \alpha - 1$, whereby we prove the statement for arbitrary $q < \alpha$.  
To this end, we present an algorithm that reduces the situation to percolation on the lattice $\L^d$.
Here, $\L^d$ is the graph with vertex set $\Z^d$, where two vertices $z_1,z_2$ are joined by an edge if and only if $\Vert z_1-z_2\Vert = 1$.
We will show that there exists a $p = p(\beta) \in [0,1]$ such that the probability $\P_\beta(C_q^\infty)$ can be bounded from below by the probability of an infinite connected component containing the origin in percolation on $\L^d$ with parameter $p$.
Percolation on $\L^d$ refers to a random subgraph of $\L^d$, where each edge is retained independently with probability $p$.
We will prove the convergence $p \to 1$ as $\beta \to \infty$.  
It follows that $\P_\beta(C_q^\infty) > 0$ for sufficiently large $\beta$, since the critical probability for percolation on the lattice $\L^d$ is strictly less than 1.

\bigskip

\noindent In the following, we assign two Poisson processes to each edge on the lattice $\L_d$, which arise by restriction from an independent marking of $\Psi$.  
Let $\tilde{\Psi}$ be an independent $\mathcal{U}(\{1,\dots,2d\})$-marking of $\Psi$.
Recall that the construction of $\Delta$ (compare Section \ref{Sec:Definition_Model}) uses a partition $\mathcal{Q}^t=(Q_i^t)_{i\in\N_0}$ of $\R^d$ into (half-open) cubes $Q_i^t:=[-t,t)^d+2t z_i$ for some $t>0$, but any numbered partition into measurable sets could be used.
For this proof we use the refined decomposition $(Q_i^{\delta_*} \times \AA \times \{l\})_{i \in \N_0, l \in \{1,\dots,2d\}}$ of $\R^d \times \AA \times \{1,\dots,2d\}$ with the parameter $\delta_* := \delta/2\sqrt{d+3}$ for the construction of the simplicial complex $\Delta$, where $\delta$ is taken from property \ref{V1} of the connection functions.
The additional labels in $\{1, \dots, 2d\}$ are solely used to refine the partition and therefore do not affect the distribution of $\Delta$.  
We say that two cubes $Q_i^{\delta_*}$ and $Q_j^{\delta_*}$ are adjacent if $z_i$ and $z_j$ are adjacent in $\L^d$, i.e., if $\Vert z_i - z_j \Vert = 1$.  
For two adjacent cubes $Q_i^{\delta_*}$ and $Q_j^{\delta_*}$, we then have  
\begin{align*}
\diam(Q_i^{\delta_*} \cup Q_j^{\delta_*}) = 2\sqrt{d+3} \delta_* = \delta,
\end{align*}  
which explains the definition of $\delta^*$.
We denote by $\tilde{\Psi}_i^l$ the restriction of $\tilde{\Psi}$ to $Q_i^{\delta_*} \times \AA \times \MM \times \{l\}$.  
For each $i \in \N_0$, there are $2d$ cubes adjacent to $Q_i^{\delta_*}$, and we assign one of the processes $\tilde{\Psi}_i^l$ to each of those cubes.  
This assignment can also be made explicit, but this is not necessary here.  
Thus, each ordered pair of adjacent cubes $(Q_i^{\delta_*}, Q_j^{\delta_*})$ is assigned exactly one process (and each edge of the lattice $\L^d$ is assigned exactly two processes).
We now define an algorithm that sequentially reveals subprocesses $\Psi_i^l$. 
By the construction of $\Delta$, at any point in time, the complex consisting of points from the processes revealed up to that point is $\P$-almost surely uniquely determined by the revealed processes.  
When we refer to revealing a process $\Psi_i^l$, we always mean that we fully uncover its projection onto $\R^d \times \AA \times \{1,\dots,2d\}$ and subsequently reveal all components of marks in $\MM$ that are necessary for the complex consisting of points from the processes revealed so far.  
Thus, infinitely many components in $\MM$ remain unrevealed at all times, which may still be used for the complex at a later stage if needed.  
To keep the notation as simple as possible, we identify the vertices and simplices of $\Delta^\0$ with their spatial components during the proof and denote the corresponding mark in $\AA$ for $x \in \Delta^\0$ by $a(x)$.
Here $\Delta^\0$ arises from $\tilde{\Psi}+\delta_{(\0,V,U,1)}$ (the additional label of the added vertex plays no role).
\begin{enumerate}
\item[(I)] Let $G_0$ be the graph with vertex set $\Z^d$ that initially has no edges, and let $\mathcal{C}_0(\0)$ be the connected component of the origin in $G_0$.  
We first reveal the process $\delta_{(\0,V,U,1)}+\sum_{l=1}^{2d}\tilde{\Psi}_0^l$ and denote by $\Delta_0^\0$ the simplicial complex resulting from this process.  
If there exist points $x_1,\dots,x_q$ in this complex such that $[\0,x_1,\dots,x_q]\in\Delta_0^\0$ and $a(x_i)\in A$ for all $i\in\{1,\dots,q\}$, and if $V \in A$, we set $\pi_0:=[\0,x_1,\dots,x_q]$.
Otherwise, we terminate the algorithm here.
\item[(II)] Let the graph $G_m$ be given, and let $\mathcal{C}_m(\0)$ be the connected component of the origin in $G_m$, with $\Delta_m^\0$ denoting the complex consisting of the points from the previously revealed processes.  
Furthermore, for each $k \in \N_0$ with $z_k \in \mathcal{C}_m(\0)$, there is a $q$-simplex $\pi_k \in \Delta_m^\0$ consisting of points in $Q_k^{\delta_*}$ with labels in $A$.  
We choose an edge $[z_i,z_j]$ of the lattice $\L^d$ with $z_i \in \mathcal{C}_m(\0)$ and $z_j \notin \mathcal{C}_m(\0)$, which has not yet been considered by the algorithm.  
If no such edge exists, we terminate the algorithm.
\item[(III)] Let $\pi_i=[y_0,\dots,y_q]$ and $l\in\{1,\dots,2d\}$ be the unique index such that $\tilde{\Psi}_j^l$ is assigned to the edge $[z_i,z_j]$.  
We reveal the process $\tilde{\Psi}_j^l$ and denote by $\Delta_{m+1}^\0$ the complex resulting from the revealed processes.  
If we find distinct points $x_0,\dots,x_q$ in the process $\tilde{\Psi}_j^l$ such that $a(x_i)\in A$ for all $i \in \{0,\dots,q\}$, and the property
\begin{align}\label{NachbarwürfelEigenschaft}
[x_0,\dots,x_{j},y_{j},\dots,y_q]\in\Delta_{m+1}^\0 \qquad \text{for all } j\in\{0,1,\dots,q\},
\end{align}
holds, we add the edge $[z_i,z_j]$ to the graph $G_m$ and denote the resulting graph by $G_{m+1}$.  
We also set $\pi_j:=[x_0,\dots,x_q]$.  
If no such points exist, we set $G_{m+1} = G_m$.
\item[(IV)] Return to step (II).
\end{enumerate}
\noindent Note that property (\ref{NachbarwürfelEigenschaft}) in step (III) implies that $\pi_i$ and $\pi_j$ in $G_q(\Delta^\0)$ are connected by an edge path.  
This algorithm determines a random subgraph of the lattice $\L^d$.  
If the algorithm terminates after finitely many steps, the subgraph, and in particular the connected component of the origin within this subgraph, is finite.  
We will show that, with sufficiently large intensity, this does not occur with positive probability.  
Due to property \ref{V1} of the connection functions, with positive probability the algorithm does not stop in the first step.  
In the following, we consider the situation where, in step (II) of the algorithm, an edge $[z_i,z_j]$ with the required properties is found.  
All objects that are revealed in step (III) are stochastically independent of all objects revealed in previous steps.  
To estimate the probability that the edge $[z_i,z_j]$ is added to the graph, independent of the previous course of the algorithm, we use the representation
\begin{align*}
\tilde{\Psi}_j^l \,\stackrel{d}{=}\, \sum_{k=1}^\tau \delta_{(X_k,V_k,U_k,l)}
\end{align*}
with independent random elements $X_i \sim (2\delta_*)^{-d}\lambda_d(\cdot\cap Q_i^{\delta_*})$, $V_i \sim \Theta$, $U_i \sim \Q$, $\tau \sim \mathrm{Po}\left(\beta \delta_*^d/d\right)$ (cf. Proposition 3.5 in \cite{Last.Lectures}).
For $I \subset \mathbb{N}$ with $|I| = q + 1$, we write $\sigma_I$ for the complex consisting of the points $X_i$ with $i \in I$. We say that such a simplex $\sigma_I$ satisfies property (\ref{NachbarwürfelEigenschaft}) if $V_i \in A$ for all $i \in I$ and there is a numbering $X_{l_0}, \dots, X_{l_q}$ of the points of $\sigma_I$ such that property (\ref{NachbarwürfelEigenschaft}) holds for $x_i = X_{l_i}$. Furthermore, we define $\sigma_k := [X_{(k-1)(q+1)+1}, \dots, X_{k(q+1)}]$, for $k \in \mathbb{N}$, and
\begin{align*}
N(l) \,:=\, \max\left\{ k \in \{0, \dots, l\} \;\|\; k(q+1) \leq l \right\}, \qquad l \in \mathbb{N}_0.
\end{align*}
Given the previous progress of the algorithm (up to step (2)), we can lower bound the probability for adding the edge $[z_i, z_j]$ independent of the previous progress of the algorithm by
\begin{align}
&\P\left( \sigma_I \text{ satisfies property } (\ref{NachbarwürfelEigenschaft}) \text{ for some } I \subset [\tau] \text{ with } |I| = q + 1 \right) \nonumber \\
&\quad \geq\, \P\left( \sigma_k \text{ satisfies property } (\ref{NachbarwürfelEigenschaft}) \text{ for some } k \in [N(\tau)] \right) \nonumber \\
&\quad =\, 1 - \P\left( \sigma_k \text{ satisfies property } (\ref{NachbarwürfelEigenschaft}) \text{ for no } k \in [N(\tau)] \right) \nonumber \\
&\quad =\, 1 - \mathbb{E} \left[ \P\left( \sigma_1, \dots, \sigma_{N(\tau)} \text{ do not satisfy property } (\ref{NachbarwürfelEigenschaft}) \mid \tau \right) \right] \nonumber \\
&\quad =\, 1 - \mathbb{E} \left[ \P\left( \sigma_1 \text{ does not satisfy property } (\ref{NachbarwürfelEigenschaft}) \mid \tau \geq q + 1 \right)^{N(\tau)} \right]. \label{Zwischenabschätzung}
\end{align}
Due to property \ref{V1} and $\diam(Q_i^{\delta_*} \cup Q_j^{\delta_*}) = \delta$, the following estimate holds
\begin{align*}
\P\left( \sigma_1 \text{ satisfies property } (\ref{NachbarwürfelEigenschaft}) \mid \tau \geq q + 1 \right) \,\geq\, \Theta(A)^{q+1} \varepsilon^{q+1} \,>\, 0.
\end{align*}
Therefore, we can lower bound the expression in (\ref{Zwischenabschätzung}) by
\begin{align*}
1 - \mathbb{E} \left[ \left( 1 - \Theta(A)^{q+1} \varepsilon^{q+1} \right)^{N(\tau)} \right]\, =:\, p.
\end{align*}
We show that $p \to 1$ as $\beta \to \infty$.
For this, we abbreviate $p_0 := 1 - \Theta(A)^{q+1} \varepsilon^{q+1}$ and $b := \beta \delta_*^d/2d$, so $\tau \sim \mathrm{Po}(b)$.
Since $p_0 < 1$, for $b \geq 1$ we have
\begin{align*}
\mathbb{E} \left[ p_0^{N(\tau)} \right] \,&=\, \sum_{k=0}^\infty e^{-b} \frac{b^k}{k!} p_0^{N(k)} \\
&=\, e^{-b} \sum_{k=0}^\infty \sum_{l=0}^q \frac{b^{k(q+1)+l}}{(k(q+1)+l)!} p_0^k \\
&\leq\, e^{-b} \sum_{k=0}^\infty \sum_{l=0}^q \frac{b^{k(q+1)+q}}{(k(q+1))!} p_0^k \\
&=\, e^{-b} (q+1) b^q \sum_{k=0}^\infty \frac{b^{k(q+1)}}{(k(q+1))!} \left(p_0^{\frac{1}{q+1}}\right)^{k(q+1)} \\
&\leq\, e^{-b} (q+1) b^q \sum_{k=0}^\infty \frac{b^{k}}{k!} \left(p_0^{\frac{1}{q+1}}\right)^{k} \\
&=\, e^{-\beta} (q+1) b^q \exp\left( b p_0^{\frac{1}{q+1}} \right) \\
&=\, (q+1) b^q \exp\left( -b\left(1 - p_0^{\frac{1}{q+1}}\right) \right) \;\rightarrow\; 0 \quad \text{as } b \to \infty,
\end{align*}
which completes the proof.
We extend the random subgraph of the lattice $\L^d$ determined by the algorithm by independently deciding for each edge fo $\L^d$ not considered by the algorithm with probability $p$ whether to add it to the random subgraph.
The resulting random graph is denoted by $G$. \\
Since, in each step of the algorithm, the objects revealed in the current step are independent of those revealed in previous steps, and since the probability of adding the considered edge is always at least $p$, we can couple $G$ with a random graph $H$ obtained from percolation on $\L^d$ with parameter $p$, such that $H \subseteq G$.
Because $p \to 1$ as $\beta \to \infty$, there exists a $\beta > 0$ such that the origin has an infinite connected component in $H$, and thus in $G$, with positive probability.
In this case, we find infinitely many simplices $\pi_i \in \Delta^\0$ in the course of the algorithm, which are connected to $\pi_0$ in $G_q(\Delta^\0)$ by an edge path.
Thus, the event $C_q^\infty$ occurs with positive probability, which proves that $\beta_c^{q} < \infty$.
\end{proof}

We would like to point out that the algorithm used in the proof of Theorem \ref{Th:kritischeIntensität} is typically (in the proofs of similar results) constructed differently and reduces the situation to percolation on the two-dimensional lattice $\L^2$.  
Examples of this are Theorem 1 in \cite{Penrose.RCM} and Lemma 2.2 in \cite{Caicedo}.
In Theorem 1 in \cite{Penrose.RCM}, the statement of Theorem \ref{Th:kritischeIntensität} for $q=0$ in the unmarked stationary case was shown under a minimal integration assumption.
An open question remains whether strict inequalities actually hold between the critical intensities from Theorem \ref{Th:kritischeIntensität}.
In \cite{Hirsch}, a conjecture (between Theorem 2 and Theorem 3) is made, the validity of which would imply strict inequalities for the critical intensities corresponding to the Vietoris-Rips complex.

\begin{Prop}
In the situation of Theorem \ref{Th:kritischeIntensität}, for $\beta>\beta_c^{(q)}$ with $q<\alpha$, it holds that
\begin{align*}
\P_\beta\left( \text{There is an infinite connected component in } G_q(\Delta) \right) = 1.
\end{align*}
\end{Prop}

\begin{proof}
Let $B_q^\infty$ denote the event that the graph $G_q(\Delta)$ percolates, i.e., has an infinite connected component.
Choose some $t>0$ and define $\Psi_i := \Psi_{Q_i^t}$ and $\A_i := \sigma(\Psi_i)$ for $i\in\N_0$.
Then, the family $(\A_i)_{i\in\N_0}$ is stochastically independent.
For each finite subset $I \subset \N_0$, the process $\sum_{i\in I}\Psi_i$ contains almost surely only finitely many points.
Since each of these points, due to \ref{V2}, almost surely has finite simplex degrees, the graph $G_q(\Delta)$ almost surely contains only finitely many vertices that intersect the set $W := \cup_{i\in I} Q_i^t$ (the vertices of $G_q(\Delta)$ are $q$-simplices).
Thus, the graph $G_q(\Delta)$ almost surely has an infinite connected component if and only if this is true for the graph $G_q(\Delta_{W^\mathrm{C}})$.
The latter depends only on $\sum_{j\in \N_0 \setminus I} \Psi_j$, which means that the event $B_q^\infty$ is independent of $(\Psi_i)_{i\in I}$.
Since this holds for every finite subset $I\subset\N_0$, Kolmogorov's 0-1 law implies that $\P_\beta(B_q^\infty) \in \{0, 1\}$.
Since $\beta > \beta_c^{(q)}$, we have $\P_\beta(B_q^\infty) > 0$, and the claim follows.
\end{proof}

\section{Sharp phase transition} \label{Sec:SharpPhaseTransition}

The goal of this section is to derive a sharp phase transition for the percolation function (which we will define precisely later), meaning a result of the form of Theorem 1 in \cite{Hirsch} or Theorem 8.2 in \cite{Last.OSSS}.  
The former result concerns percolation of $\tilde{G}_q(\Delta)$ of the Vietoris-Rips complex and is based on the discrete OSSS inequality, which was first proved in \cite{ODonnell} (Theorem 3.1) and is also detailed in \cite{Duminil-Copin} (Theorem 1.9), where it is described in a manner directly applicable to our situation.
In \cite{Last.OSSS}, a continuous version of the OSSS inequality for functionals of a Poisson process is derived using decision trees in continuous time, which is then used to establish a sharp phase transition for $k$-percolation in the Boolean model (see Chapter 8 there). \\
Our strategy is to adapt the application of the OSSS inequality, as demonstrated in \cite{Duminil-Copin,Hirsch}, for proving a sharp phase transition in percolation models, to our specific scenario.
In \cite{Duminil-Copin} percolation in the Boolean model with balls as grains is examined, which, as we will demonstrate (cf. Section \ref{Sec:ExampleModels}), constitutes a special case of the situation we address here, similar to the Vietoris-Rips complex considered in \cite{Hirsch}.
In these cases, the connection functions take values only in $\{0,1\}$, which simplifies the situation compared to general connection functions.
Whereas \cite{Hirsch} deals with the unmarked scenario and \cite{Duminil-Copin} uses the mark space $\R_{\geq 0}$, our work involves a general mark space $\AA$. \\  
Throughout this section, let a fixed $q\in\{0,\dots,\alpha-1\}$ be given, and assume that the connection functions always satisfy properties \ref{V1} and \ref{V2} for this fixed $q$ with parameters $\delta,\varepsilon,D>0$ and $A\subseteq\AA$. 
For the construction of the simplicial complex $\Delta$, we choose $t=D$ (cf. Section \ref{Sec:Definition_Model} for a detailed description of the construction), which means that we decompose $\R^d$ into cubes of edge length $2D$.
Recall that $\Delta^\0$ denotes the simplicial complex constructed from $\Psi+\delta_{(\0,V,U)}$.  
Since we will always work with $\Delta^\0$ in this section, we set $\Delta:=\Delta^\0_{\R^d\setminus\{\0\}}$ to avoid unnecessary complications in notation.  
Note that this does not change the distribution of $\Delta$.  
Now, we introduce the percolation events necessary for proving the sharp phase transition.  
For two simplices $\sigma,\rho\in F_q(\Delta^\0)$, we write $\sigma\xleftrightarrow{q}\rho$ as shorthand if they are connected by an edge path in $G_q(\Delta^\0)$.  
For $x\in\R^d$, $B\in\mathcal{B}(\R^d)$, $r>0$, and a simplicial complex $K$, we define  
\begin{align*}
x\xleftrightarrow{K} B &\,:=\, \big\{ \exists\sigma,\rho\in F_q(K) \text{ with } \sigma\xleftrightarrow{q}\rho, x\in\sigma, \diam(\sigma)\leq D, \rho\cap B\neq\emptyset \big\}, \\
x\leftrightarrow B &\,:=\, x\xleftrightarrow{\Delta} B, \\
x\rightleftarrows B &\,:=\, (x\leftrightarrow B) \cap (x\leftrightarrow B^c), \\
B_r &\,:=\, \0\xleftrightarrow{\Delta^\0} B(\0,r)^c, \\
\theta_r(\beta) &\,:=\, \P( B_r ), \\
\theta_\infty(\beta) &\,:=\, \P\big( \exists\sigma\in F_q(\Delta^\0) \text{ with an infinite component in } G_q(\Delta^\0) \text{ and }\0\in\sigma \big).
\end{align*} 
The event $B_r$ can be interpreted as the existence of an edge path from the origin to $B(\0,r)^c$ in $G_q(\Delta^\0)$.  
Note that simplices $\sigma\in F_q(\Delta^\0)$ with $\diam(\sigma)>D$ form $\P$-almost surely isolated vertices in $G_q(\Delta^\0)$ due to property \ref{V2} of the connection functions and are therefore irrelevant for the percolation of $\Delta^\0$.  
In defining the event $B_r$ for $r>D$, we explicitly exclude the case where there exists a simplex $\sigma\in F_q(\Delta^\0)$ with $\0\in\sigma$ and $\sigma\cap B(\0,r)^c\neq\emptyset$ (this would imply $\diam(\sigma)>D)$.  
As a result, due to the construction used for $\Delta^\0$, the event $B_r$ (at least $\P$-almost surely) depends only on $\Psi_{W^{(r)}}+\delta_{(\0,V,U)}$, where $W^{(r)}$ is the union of all cubes $Q\in\mathcal{Q}^D$ with $Q\cap B(\0,r+D)\neq\emptyset$.
In particular $W$ is bounded.
The notation $x\rightleftarrows B$ serves the purpose of avoiding a case distinction based on whether $x\in B$ or $x\notin B$, while keeping in mind that this notation always refers to the complex $\Delta$ and not to $\Delta^\0$.
Finally, we point out that  
\begin{align*}
\theta_\infty(\beta) \,=\, \lim_{r\rightarrow\infty} \theta_r(\beta)
\end{align*}  
holds.
We call the function $\theta:(0,\infty]\times(0,\infty)\rightarrow[0,1], (r,\beta)\mapsto\theta_r(\beta)$ the percolation function of the model.
To derive a sharp phase transition for $q$-percolation, we define an algorithm that determines the occurrence of the event $B_r$ and to which we can then apply the OSSS inequality.
For this, let $d(B_1,B_2):=\inf\{ \Vert x-y \Vert \,:\, x\in B_1, y\in B_2 \}$ for $B_1,B_2\in\mathcal{B}(\R^d)$.

\begin{Alg}
Let $0<s\leq r$ and  
\begin{align*}
I_r \,:=\, \big\{i\in\N_0 \,\|\, Q_i^D\cap B(\0,r+D)\neq\emptyset\big\}, \qquad W^{(r)}\,:=\,\cup_{i\in I_r} Q_i^D.
\end{align*}
We decompose the process $\Psi_{W^{(r)}}+\delta_{(\0,V,U)}$ into the subprocesses
\begin{align*}
\Psi_0 \,&:=\, \Psi_{Q_0^D} + \delta_{(\0,V,U)}, \qquad \Psi_i \,:=\, \Psi_{Q_i^D}, \qquad i\in I_r\setminus\{\0\}.
\end{align*}
In the following, we say that a connected component of $G_q(\Delta_W)$ for some $W\subset\R^d$ intersects the sphere $\partial B(\0,s)$ if there exist simplices $\sigma,\rho$ in this connected component with $\sigma\cap B(\0,s)\neq\emptyset$ and $\rho\cap B(\0,s)^c\neq\emptyset$.
Define an algorithm that reveals one of the processes $\Psi_i$, $i\in I_r$, step by step, as follows.
\begin{enumerate}
\item[(I)] Set $T_0:=\{ i\in I_r \,|\, d(Q_i^D,\partial B(\0,s))\leq D \}$ and $W_0:=\cup_{i\in T_0} Q_i^D$.  
Reveal all Poisson processes $\Psi_i$ with $i\in T_0$ and denote by $\mathcal{C}_0$ the set of all $q$-simplices in $\Delta_{W_0}$ with diameter at most $D$ whose connected component in $G_q(\Delta_{W_0})$ intersects the sphere $\partial B(\0,s)$.  
\item[(II)] Let $T_m\subseteq I_r$ be the set of indices of all previously revealed processes, $W_m:=\cup_{i\in T_m} Q_i^D$, and $\mathcal{C}_m$ the set of all $q$-simplices in $\Delta_{W_m}$ with diameter at most $D$ whose connected component in $G_q(\Delta_{W_m})$ intersects the sphere $\partial B(\0,s)$.  
\begin{itemize}
\item If there exists a cube $Q_i^D$ with $i\in I_r\setminus T_m$ such that there is a simplex $\sigma$ in $\mathcal{C}_m$ with $d(\sigma,Q_i^D)\leq D$, then set $T_{m+1}:=T_m\cup\{i\}$ and $W_{m+1}:=\cup_{i\in T_{m+1}} Q_i^D$.  
Reveal the Poisson process $\Psi_i$ and denote by $\mathcal{C}_{m+1}$ the set of all $q$-simplices in $\Delta_{W_{m+1}}$ with diameter at most $D$ whose connected component in $G_q(\Delta_{W_{m+1}})$ intersects the sphere $\partial B(\0,s)$.  
\item If no such cube exists, stop the algorithm.  
\end{itemize}
\end{enumerate}
\end{Alg}

\begin{figure}[t]
  \captionsetup{ labelfont = {bf}, format = plain }
  \centering
  \includegraphics[width=0.7\linewidth]{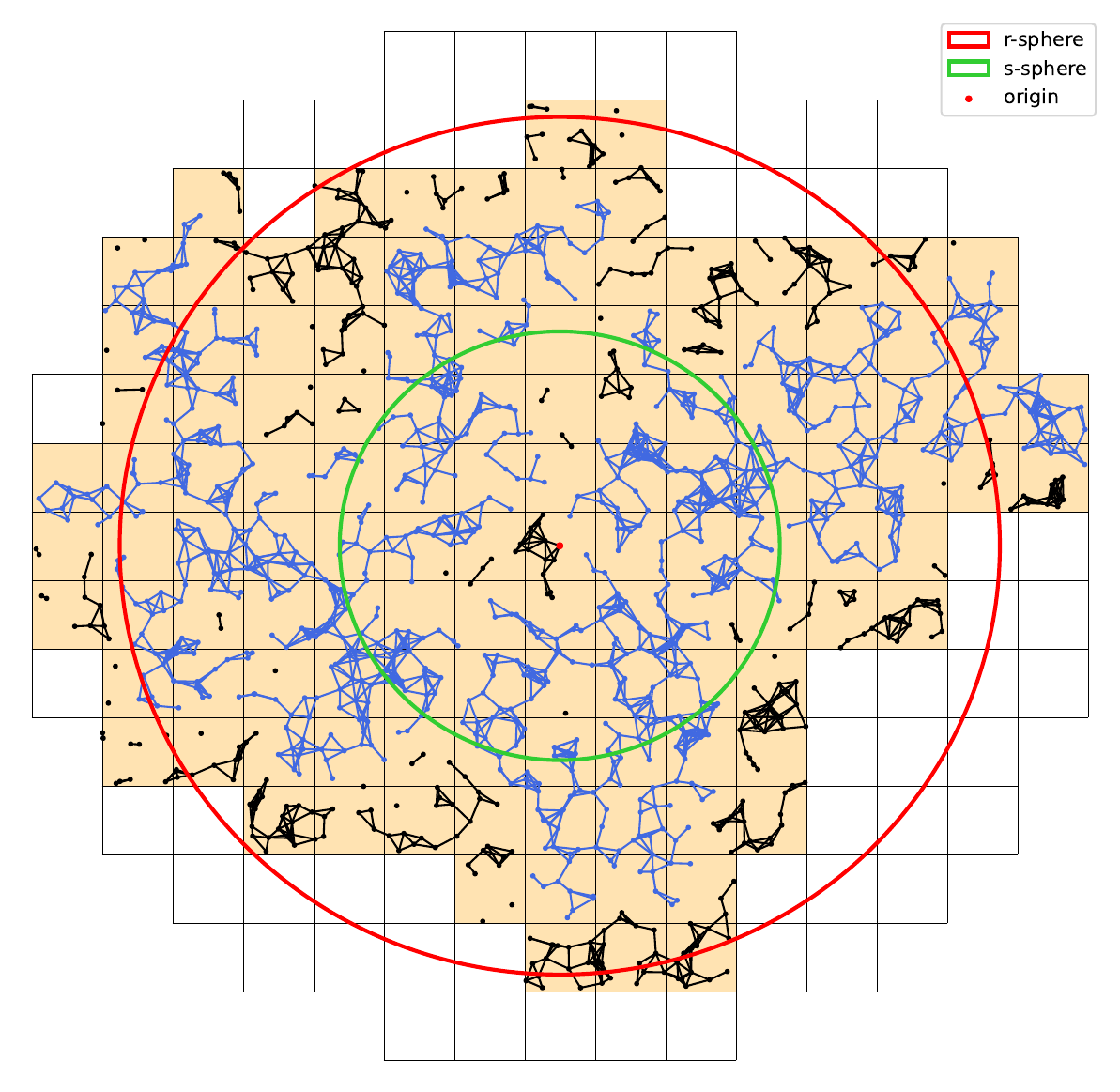}
  \caption{A realisation of the algorithm in the unmarked case for $q=0$, $\varphi_1(x,y)=\1\{ \Vert x-y \Vert\leq r_0 \}$ with $r_0=\frac{3}{5}$, $r=10$, $s=5$, $D=\frac{4}{5}$, and $\beta=4$, where the connected components of $\Delta_{W^{(r)}}^\0$ that intersect the sphere $\partial B(\0,s)$ are shown in blue.
  The event $B_r$ has not occurred in this realisation.}
  \label{fig:OSSS}
\end{figure}

Clearly, the algorithm stops after finitely many steps due to $|I_r|<\infty$.
The crucial point is that, $\P$-almost surely, the occurrence of the event $B_r$ is determined by the Poisson processes revealed up to the stopping time of the algorithm.
Note that the sets $\mathcal{C}_m$ from the algorithm always refer to the complex $\Delta_{W_m}$ and not to $\Delta_{W_m}^\0$, and the algorithm reveals $\P$-almost surely all connected components in $\Delta_{W^{(r)}}$ that intersect the sphere $\partial B(\0,s)$.
Since all simplices $\sigma \in F_q(\Delta^\0)$ with $\0 \in \sigma$ and $\diam(\sigma) \leq D$ and all their neighbors in $G_q(\Delta^\0)$ have $\P$-almost surely vertices in $Q_0^D \times \AA$ due to property \ref{V2} of the connection functions, the process $\Psi_0$ is revealed in the case of the occurrence of $B_r$.
Figure \ref{fig:OSSS} shows a realisation of the algorithm for the geometric graph, where the event $B_r$ did not occur.
The OSSS inequality (Theorem 1.9 in \cite{Duminil-Copin}) now gives
\begin{align} \label{OSSS}
\theta_r(\beta)(1-\theta_r(\beta)) \,\leq\, \sum_{i\in I_r} \delta_i(s) \, \zeta_i
\end{align}
with
\begin{align*}
\delta_i(s) &\,:=\, \P\big( \Psi_i \text{ is revealed by the algorithm} \big), \\
\zeta_i &\,:=\, \P\Big( \1\big\{\Psi_{W^{(r)}}+\delta_{(\0,V,U)}\in B_r\big\} \neq \1\big\{\tilde{\Psi}_i\in B_r\big\} \Big),
\end{align*}
where $\tilde{\Psi}_i$ is a Poisson process arising from $\Psi_{W^{(r)}}+\delta_{(\0,V,U)}$ by replacing $\Psi_i$ with an independent copy.
Note that only $\delta_i(s)$ depends on $s$. To prove a sharp phase transition, we need upper bounds for these two quantities that allow us to apply the Margulis-Russo formula.
To this end, we formulate two lemmas that provide these upper bounds.

\begin{Lem} \label{Lem:SimWkeit}
For $i\in I_r$, we have
\begin{align*}
\int_0^r \delta_i(s) \;\d s \,\leq\, 2D(1+\sqrt{d}) + 2\beta(4D)^d \int_0^r \theta_s \;\d s.
\end{align*}
\end{Lem}

\begin{proof}
Fix $i\in I_r$ and $0<s\leq r$, and write $B_1+B_2:=\{b_1+b_2\,|\,b_1\in B_1, b_2\in B_2\}$ for the Minkowski sum of two sets $A,B\in\mathcal{B}(\R^d)$.
Furthermore, we write $\tilde{z}_i:=2Dz_i$ for the center of $Q_i^D$.
If $|\Vert\tilde{z}_i\Vert-s|>D(1+\sqrt{d})$, then $\Psi_i$ is not revealed in the first step of the algorithm.
In this case, we have
\begin{align*}
\delta_i(s) \,&\leq\, \P\big( \exists (x,a)\in \Phi_{Q_i+B(\0,D)} \text{ with } x\rightleftarrows B(\0,s) \big) \\  
&\leq\, \P\big( \exists (x,a)\in \Phi_{Q_i+B(\0,D)} \text{ with } x\rightleftarrows B(\tilde{z}_i,|\Vert\tilde{z}_i\Vert-s|) \big) \\  
&\leq\, \P\Big( \exists (x,a)\in \Phi_{Q_i+B(\0,D)} \text{ with } x\rightleftarrows B\big(x,|\Vert\tilde{z}_i\Vert-s|-D(1+\sqrt{d})\big) \Big) \\  
&\leq\, \mathbb{E} \bigg[ \, \sum_{(x,a)\in\Phi_{Q_i+B(\0,D)}} \, \1\Big\{ x\rightleftarrows B\big(x,|\Vert\tilde{z}_i\Vert-s|-D(1+\sqrt{d})\big) \Big\}\, \bigg] \\  
&=\, \beta \int_{Q_i+B(\0,D)} \theta_{|\Vert\tilde{z}_i\Vert-s|-D(1+\sqrt{d})} \; \d x \\  
&\leq\, \beta (4D)^d \;\theta_{|\Vert\tilde{z}_i\Vert-s|-D(1+\sqrt{d})},
\end{align*}
where the Mecke equation (for the Poisson process $\Psi$) was applied in the penultimate step.
In the case of $\Vert \tilde{z}_i \Vert >s$ the second and third estimates are illustrated by Figure \ref{fig:Perkolationslemma}.
Since the interior of the ball $B(\0,s)$ lies in the complement of the ball $B(\tilde{z}_i,|\Vert\tilde{z}_i\Vert-s|)$, the condition $x\rightleftarrows B(\0,s)$ implies $x\rightleftarrows B(\tilde{z}_i,|\Vert\tilde{z}_i\Vert-s|)$.
Furthermore, for $x\in Q_i^D+B(\0,D)$, we always have $\Vert x-\tilde{z}_i \Vert\leq D(1+\sqrt{d})$.
Therefore, it follows that $B\big(x,|\Vert\tilde{z}_i\Vert-s|-D(1+\sqrt{d})\big)\subseteq B(\tilde{z}_i,|\Vert\tilde{z}_i\Vert-s|)$, and thus, from $x\rightleftarrows B(\tilde{z}_i,|\Vert\tilde{z}_i\Vert-s|)$, we also get $x\rightleftarrows B\big(x,|\Vert\tilde{z}_i\Vert-s|-D(1+\sqrt{d})\big)$.
The case $\Vert \tilde{z}_i \Vert <s$ works analogously. \\  
Using the trivial bound $\delta_i(s)\leq 1$ for $|\Vert\tilde{z}_i\Vert-s|\leq D(1+\sqrt{d})$, integration over $s\in (0,r)$ yields the desired statement.
\end{proof}

\begin{figure}[t]
    \captionsetup{
    labelfont = {bf},
    format = plain,
  }
	\centering
	\begin{normalsize}
	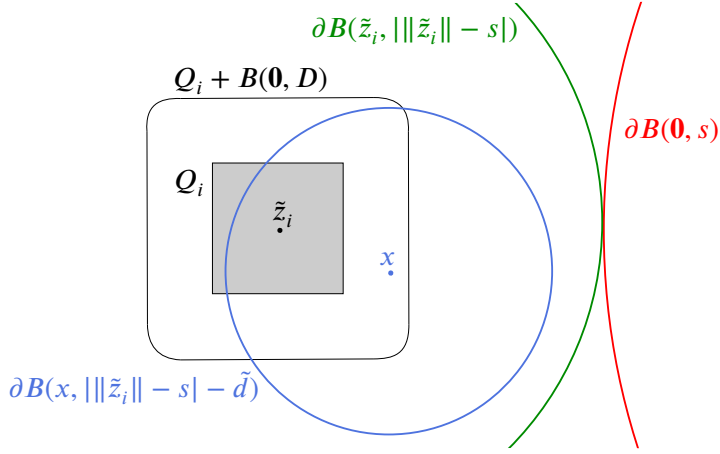
	\end{normalsize}
		 
	\caption{A visualisation of the argumentation from Lemma \ref{Lem:SimWkeit} in the case $\Vert \tilde{z}_i \Vert >s$, where we abbreviate $\tilde{d}:=D(1+\sqrt{d})$.}
	\label{fig:Perkolationslemma}
\end{figure}

\noindent In the next step, we need a suitable upper bound for $\zeta_i$.
To simplify the notation, we write $\eta \in B_r$ if for a counting measure or a point process $\eta$, it holds that in the simplicial complex constructed from $\eta$, the origin is connected to the complement of $B(\0,r)$ in the sense of the definition of $B_r$.
Furthermore let
\begin{align*}
    D_{(x,a,u)}f(\Psi) \,:=\, f(\Psi+\delta_{(x,a,u)})-f(\Psi)
\end{align*}
be the difference operator for a real-valued function $f$ and $(x,a,u)\in\R^d\times\AA\times\MM$.

\begin{Lem} \label{Lem:Einfluss}
For $i \in I_r$, we have
\begin{align*}
\zeta_i \,\leq\, 2\beta e^{\beta(2D)^d} \int_{Q_i^D}\int_\AA\int_\MM \E{ D_{(x,a,u)} f(\Psi) } \; \Q(\d u) \; \Theta(\d a) \; \d x
\end{align*}
with $f:\mathds{N}(\R^d\times\AA\times\MM)\rightarrow\R$,
\begin{align*}
f(\eta) \,:=\, \E{ \1\big\{ \eta+\delta_{(\0,V,U)}\in B_r \big\} }.
\end{align*}
\end{Lem}

\begin{proof}
First, we define
\begin{align*}
\hat{\Psi} \,&:=\, \sum_{j \in I_r} \Psi_j \,=\, \Psi + \delta_{(\0,V,U)}, \qquad \bar{\Psi}_i \,:=\, \sum_{j \in I_r \setminus \{i\}} \Psi_j, \\
\tilde{\Psi}_i \,&:=\, \bar{\Psi}_i + \sum_{l=1}^\vartheta \delta_{(Y_l,V_l,U_l)} + \1\{ i=0 \} \delta_{(\0,\tilde{V},\tilde{U})}
\end{align*}
with $\vartheta \sim \mathrm{Po}(\beta(2D)^d)$, uniformly distributed random points $Y_1,Y_2,\dots$ in $Q_i^D$, and marks $\Theta \sim \tilde{V},V_1,V_2,\dots$ and $\Q \sim \tilde{U},U_1,U_2,\dots$, where all these objects are independent (in particular also from $\Psi,V,U$).
Then, we have
\begin{align}
\zeta_i \,&=\, \P\Big( \hat{\Psi}\in B_r, \,\tilde{\Psi}_i\notin B_r \text{ or } \hat{\Psi}\notin B_r, \,\tilde{\Psi}_i\in B_r \Big) \nonumber \\
&=\, 2\, \P\Big( \hat{\Psi}\notin B_r, \tilde{\Psi}_i\in B_r \Big) \nonumber \\
&\leq\, 2\, \P\Big( \bar{\Psi}_i\notin B_r, \,\tilde{\Psi}_i\in B_r \Big). \label{Influence}
\end{align}
For the inequality, note that $\P$-almost surely, the simplicial complex constructed from $\bar{\Psi}_i$ is contained in the simplicial complex constructed from $\hat{\Psi}$, and therefore $\hat{\Psi}\in B_r$ implies $\bar{\Psi}_i\in B_r$.
For the further proof, we need a new construction.
We extend our basic space $\R^d \times \AA$ to $\R^d \times \AA \times \N_0$ and choose an arbitrary probability measure $\mathbb{S}$ on $\N_0$.
On this space, we define the decomposition $\mathcal{R} = (Q_i^D \times \AA \times \N_0)_{i \in \N_0}$ and the relation
\begin{align*}
(x,a,n) \;\,\tilde{\prec}\;\, (y,b,m) \quad :\Longleftrightarrow \quad
\begin{cases}
n<m, & x,y \in Q_i^D, \\
x < y, & \text{otherwise}, \\
\end{cases}
\end{align*}
\noindent where $<$ again denotes the lexicographical order on $\R^d$.
Furthermore, let $\bar{\xi}_i$ be a point process on $\R^d \times \AA \times \MM\times \N_0$, derived from $\bar{\Psi}_i$ by assigning to each point independently of all others a random element in $\N_0$ with distribution $\mathbb{S}$.
We consider the complex $\Delta_k$, $k\in\N_0$, constructed from
\begin{align*}
\bar{\xi}_i + \1\{i=0\}\delta_{(\0,V,U,0)} + \sum_{j=1}^k \delta_{(Y_j,V_j,U_j,j)}
\end{align*}
where the vertices of $\Delta_k$ are elements of $\R^d \times \AA$, i.e., the additional marks in $\N_0$ are forgotten after the construction of the complex.
Then, $\Delta_\vartheta$ has the same distribution as the complex constructed from $\tilde{\Psi}_i$, and together with (\ref{Influence}), it follows
\begin{align} \label{TransitionConstruction}
\frac{1}{2}\zeta_i \,\leq\, \P\left( \Delta_0\notin B_r, \Delta_\vartheta\in B_r \right).
\end{align}
For $i \neq 0$, $\Delta_0$ arises from $\Delta_k$ by restriction to $W \setminus Q_i^D$.
In the case $i = 0$, however, we have
\begin{align*}
\P\Big( \bar{\Psi}_i\notin B_r, \,\tilde{\Psi}_i\in B_r \Big) \,=\, \P\Big( \tilde{\Psi}_i\in B_r \Big) \,=\, \P\Big( \Delta_\vartheta\in B_r \Big) \,=\, \P\left( \Delta_0\notin B_r, \Delta_\vartheta\in B_r \right),
\end{align*}
since the origin in $\Delta_0$ for $q>0$ is only part of simplices with diameter larger than $D$, and in the case $q=0$, due to \ref{V2}, $\P$-almost surely forms an isolated corner.
By construction, we have $\Delta_k \subseteq \Delta_l$ for all $k \leq l$.
With the definition
\begin{align*}
N &:=\, \inf \big\{ k \in \N_0 \,\|\, \Delta_k \in B_r \big\} \in \N_0 \cup \{\infty\}
\end{align*}
($\inf\emptyset := \infty$), we obtain from (\ref{TransitionConstruction})
\begin{align*}
\frac{1}{2}\zeta_i \,&\leq\, \P( 0<N\leq\vartheta ) \\
&\leq\, \beta(2D)^d e^{\beta(2D)^d} \;\P\big( N=\vartheta+1 \big),
\end{align*}
where the second inequality follows as in \cite{Hirsch} (proof of Lemma 6).
No properties of the distribution of $N$ are used here, only the independence of $N$ and $\vartheta$.
Finally, we have
\begin{align*}
\P(N=\vartheta+1) \,&=\, \P(\Delta_\vartheta\notin B_r, \Delta_{\vartheta+1}\in B_r) \\
&=\, \E{ \1\{ \Delta_{\vartheta+1}\in B_r \} - \1\{ \Delta_\vartheta\in B_r \} } \\
&=\, \E{ \1\big\{ \hat{\Psi}+\delta_{(Y_1,V_1,U_1)} \in B_r \big\} - \1\big\{ \hat{\Psi}\in B_r \big\} } \\
&=\, \frac{1}{(2D)^d} \int_{Q_i^D}\int_\AA\int_\MM \E{ D_{(x,a,u)} f(\Psi) } \; \Q(\d u) \; \Theta(\d a) \; \d x.
\end{align*}
\end{proof}

\noindent Using the two Lemmas \ref{Lem:SimWkeit} and \ref{Lem:Einfluss}, we can now prove the sharp phase transition for the percolation function $\theta$.

\begin{Th} [Sharp Phase Transition] \label{Th:scharferPhasenübergang}
The connection functions $\varphi_1,\dots,\varphi_{q+1}$ satisfy the two properties \ref{V1}, \ref{V2}.
Then, $\beta_c^{(q)} \in (0, \infty)$, and the following two statements hold.
\begin{enumerate}
\item For all $\beta < \beta_c^{(q)}$, there exists a $c(\beta) > 0$ such that
\begin{align*}
\theta_r(\beta) \,\leq\, e^{-c(\beta)r} \qquad \text{for all } r>0.
\end{align*}
\item For all $\beta_0 > \beta_c^{(q)}$, there exists a $c(\beta_0) > 0$ such that
\begin{align*}
\theta_\infty(\beta) \,\geq\, c(\beta_0)(\beta-\beta_c^{(q)}) \qquad \text{for all } \beta \in (\beta_c^{(q)}, \beta_0).
\end{align*}
\end{enumerate}
\end{Th}

\begin{proof}
To see that $\beta_c^{(q)} \in (0, \infty)$, we replace the edge function $\varphi_1$ by the modified function $\tilde{\varphi}$ from (\ref{Modif:Kantenfkt}) and consider the connection functions $\tilde{\varphi}_1,\varphi_2,\dots,\varphi_{q+1}$.
Due to \ref{V2}, $q$-simplices with diameter greater than $D$ in $G_q(\Delta^\0)$ almost surely form only isolated vertices and thus have no effect on the event $C_q^\infty$.
Thus, the critical intensity $\beta_c^{(q)}$ remains unchanged by this modification.
However, with this edge function $\tilde{\varphi}_1$ condition \ref{V1} is satisfied for $q=0$, and Theorem \ref{Th:kritischeIntensität} therefore gives $\beta_c^{(q)} \in (0, \infty)$. \\
The two statements (i) and (ii) follow from the Lemmas \ref{Lem:SimWkeit} and \ref{Lem:Einfluss} in a similar fashion as Theorem 1 follows from Lemmas 5 and 6 in \cite{Hirsch}, where we roughly outline the main steps (especially those that use the properties of the percolation model).
First, Lemma \ref{Lem:Einfluss} and the Margulis-Russo formula for Poisson processes (Theorem 19.4 in \cite{Last.Lectures}) yield
\begin{align}
\sum_{i\in I_r} \zeta_i \,&\leq\, 2\beta e^{\beta(2D)^d} \int_{W^{(r)}}\int_\AA\int_\MM \E{ D_{(x,a,u)}f(\Psi) } \; \Q (\d u) \; \Theta (\d a) \; \d x \nonumber \\
&=\, 2\beta e^{\beta(2D)^d} \frac{\mathrm{d}}{\mathrm{d} \beta} \mathbb{E}_\beta \left[ f(\Psi) \right] \nonumber \\
&=\, 2\beta e^{\beta(2D)^d}\frac{\mathrm{d}}{\mathrm{d} \beta} \theta_r(\beta). \label{Abschätzung_Einfluss}
\end{align}
Integrating both sides of the OSSS inequality (\ref{OSSS}) over $s \in [0, r]$, we obtain with Lemma \ref{Lem:SimWkeit} and (\ref{Abschätzung_Einfluss})
\begin{align*}
r\theta_r(\beta)(1-\theta_r(\beta)) \,&\leq\, 4\beta e^{\beta(2D)^d} \left(D(1+\sqrt{d}) + \beta(4D)^d \int_0^r \theta_s(\beta) \;\d s \right) \, \frac{\mathrm{d}}{\mathrm{d}\beta} \theta_r(\beta),
\end{align*}
and thus the differential inequality
\begin{align} \label{DiffIneq}
\frac{\mathrm{d}}{\mathrm{d}\beta} \theta_r(\beta) \,\geq\, \frac{r\theta_r(\beta)(1-\theta_r(\beta))}{4\beta e^{(2D)^d\beta} \left(D(1+\sqrt{d}) + \beta(4D)^d \int_0^r \theta_s(\beta) \;\d s \right)}.
\end{align}
We fix $0<\beta_1<\beta_2<\infty$ and choose a $\beta \in (\beta_1, \beta_2)$.
For $r \geq D$, the occurrence of $B_r$ almost surely implies the existence of a $(q+1)$-simplex $\sigma \in F_{q+1}(\Delta^\0)$ with $\0 \in \sigma$, which is involved in at least one edge in $G_q(\Delta^\0)$.
Thus, the origin must almost surely be contained in a $(q+1)$-simplex for the occurrence of $B_r$, which, due to \ref{V2}, implies almost surely $\Phi(B(\0, D)) \geq q+1$.
Therefore, for $r \geq D$, we obtain
\begin{align*}
1-\theta_r(\beta) \,\geq\, \P_\beta\big( \Phi(B(\0,D))\leq q \big) \,\geq\, \P_{\beta_2}\big( \Phi(B(\0,D))\leq q \big) \,=:\, C_1 \,>\, 0.
\end{align*}
On the other hand, for $r \leq \frac{\delta}{4}$, the existence of a $(q+1)$-simplex $\sigma \in \Delta^\0$ with $\0 \in \sigma \subset B(\0, \frac{\delta}{2})$ and $\sigma \cap B(\0, \frac{\delta}{4})^c \neq \emptyset$ implies the occurrence of $B_r$.
Therefore, by property \ref{V1} of the connection functions, we have
\begin{align*}
\theta_r(\beta) \,&\geq\, \varepsilon \;\P_\beta\big( \Psi\big(B(\0,\tfrac{\delta}{4})\times A \big)\geq q, \, \Psi\big( (B(\0,\tfrac{\delta}{2})\setminus B(\0,\tfrac{\delta}{4})) \times A \big)\geq 1 \big) \\
&\geq\, \varepsilon \;\P_{\beta_1}\big( \Psi\big(B(\0,\tfrac{\delta}{4})\times A \big)\geq q, \, \Psi\big( (B(\0,\tfrac{\delta}{2})\setminus B(\0,\tfrac{\delta}{4})) \times A \big)\geq 1 \big) \\
&\geq\, \varepsilon \;\P_{\beta_1}\big( \Psi\big(B(\0,\tfrac{\delta}{4})\times A \big)= q, \, \Psi\big( (B(\0,\tfrac{\delta}{2})\setminus B(\0,\tfrac{\delta}{4})) \times A \big)= 1 \big) \\
&\geq\, \varepsilon\,\Theta(A)^{q+1}\; \P_{\beta_1}\big( \Phi\big( B(\0,\tfrac{\delta}{4} \big) = q, \, \Phi\big( B(\0,\tfrac{\delta}{2})\setminus B(\0,\tfrac{\delta}{4}) \big)= 1 \big) \,=:\, C_2 \,>\, 0.
\end{align*}
Therefore, for all $r \geq \frac{\delta}{4}$, we have
\begin{align*}
\int_{0}^r \theta_s(\beta) \;\d s \,\geq\, C_2\frac{\delta}{4}.
\end{align*}
Since the two lower bounds $C_1, C_2$ do not depend on $r \geq D$ and $\beta \in (\beta_1, \beta_2)$, the differential inequality (\ref{DiffIneq}) implies that for all $0 < \beta_1 < \beta_2 < \infty$, there exists a constant $c > 0$ such that for all $\beta \in (\beta_1, \beta_2)$ and $r \geq D$, the inequality
\begin{align*}
\frac{\mathrm{d}}{\mathrm{d}\beta} \log(\theta_r(\beta)) \,\geq\, c \frac{r}{\int_0^r \theta_s(\beta) \;\d s}
\end{align*}
holds.
This statement serves as the analogue of Lemma 4 in \cite{Hirsch}.
We define
\begin{align*}
\tilde{\beta} \,:=\, \sup \bigg\{ \beta > 0 \;\bigg|\; \limsup_{r \rightarrow \infty} \frac{\log\left(\int_0^r \theta_s(\beta)\;\d s\right)}{\log(r)} < 1 \bigg\} \in [0, \infty]
\end{align*}
with $\sup \emptyset := 0$.
From $\beta > \beta_c^{(q)}$ it follows that $\theta(\beta) > 0$ and therefore
\begin{align*}
\frac{\log\left(\int_0^r \theta_s(\beta)\;\d s\right)}{\log(r)} \,\geq\, \frac{\log(\theta(\beta)) + \log(r)}{\log(r)} \,\rightarrow\, 1 \enskip \text{as} \enskip r \rightarrow \infty
\end{align*}
and thus in particular $\tilde{\beta} \leq \beta_c^{(q)} < \infty$.
In what follows, we denote the percolation function by $\theta^{(q)}$ and define $\theta^{(l)}$ for $l \in \{0, \dots, q-1\}$ analogously (although $\kappa_l$ may not satisfy property \ref{V2}).
It is straightforward to conform that $\theta^{(l)} \leq \theta^{(k)}$ holds for $k \leq l$.
To verify $\tilde{\beta} > 0$, we once again replace the edge function $\varphi_1$ with the modification $\tilde{\varphi}_1$ from (\ref{Modif:Kantenfkt}) and denote the percolation functions related to the connection functions $\tilde{\varphi}_1, \varphi_2, \dots, \varphi_{q+1}$ by $\tilde{\theta}^{(0)}, \dots, \tilde{\theta}^{(q)}$.
Since the event $B_r$ depends only on simplices with diameter at most $D$ are relevant, it follows that $\tilde{\theta}^{(q)} = \theta^{(q)}$ and thus in particular $\theta^{(q)} \leq \tilde{\theta}^{(0)}$.
Note that in the case $q = 0$, due to property \ref{V2} (for the function $\kappa_{q+1} = \kappa_1 = \varphi_1$), the functions $\varphi_1, \tilde{\varphi}_1$ agree almost everywhere with respect to $(\lambda_d \otimes \Theta)^2$.
Clearly, the functional estimate $\tilde{\varphi}_1 \leq \hat{\varphi}$ holds true for
\begin{align*}
\hat{\varphi}\big( (x, a), (y, b) \big) \,:=\, \1\big\{ \Vert x - y \Vert \leq D \big\}.
\end{align*}
We consider the random graph with the edge function $\hat{\varphi}$, i.e., the geometric graph with parameter $D$, and denote the associated percolation function by $\hat{\theta}^{(0)}$.
Theorem \ref{Th:kritischeIntensität} gives $\hat{\beta}_c^{(0)} \in (0, \infty)$ for the corresponding critical intensity $\hat{\beta}_c^{(0)}$.
It is known (see \cite{Penrose.Graphs} or Theorem 1.1 in \cite{Ziesche} for a more general result) that the percolation function $\hat{\theta}^{(0)}$ decays exponentially in the radius for $\beta < \hat{\beta}_c^{(0)}$.
Since $\tilde{\varphi}_1 \leq \hat{\varphi}$, this also holds for $\tilde{\theta}^{(0)}$ and since $\theta^{(q)} \leq \tilde{\theta}^{(0)}$, it follows in particular for $\theta^{(q)}$.
This in turn implies that $\tilde{\beta} \geq \hat{\beta}_c^{(0)} > 0$.
The remaining proof, which includes the proof of $\tilde{\beta} = \beta_c^{(q)}$, proceeds in the same manner as the proof of Theorem 1 in \cite{Hirsch} and is purely analytical in nature.
No further properties of the percolation model are used beyond the already shown properties of the percolation function, which is why we refer entirely to \cite{Hirsch} for the remainder of the proof.
Note that the function $T_n(\beta)$ from \cite{Hirsch} must be defined here as
\begin{align*}
T_n(\beta) \,:=\, \frac{1}{\log(n)} \sum_{k=\lceil D \rceil}^n \frac{\theta_k(\beta)}{k}, \qquad n \geq \lceil D \rceil, \beta > 0,
\end{align*}
(in \cite{Hirsch}, the sum starts from $k = 1$) so that the estimates used there can be applied.
The proof works in exactly the same way with this modification, because in particular $\lim_{n \rightarrow \infty} T_n(\beta) = \theta_\infty(\beta)$.
\end{proof}

\noindent For the proof of Theorem \ref{Th:scharferPhasenübergang}, the boundedness of the connection functions in the sense of \ref{V2} is essential.
An example of a theorem like Theorem \ref{Th:scharferPhasenübergang} without a form of boundedness can be found in \cite{Duminil-Copin}.
There, condition (ii) from Theorem \ref{Th:scharferPhasenübergang} is proven for percolation in the Boolean model with balls and unbounded radius distribution, and it is noted that condition (i) in this situation generally does not hold.
As far as we know, even for $q = 0$ (i.e., for the RCM as a random graph), the sharp phase transition from Theorem \ref{Th:scharferPhasenübergang} has not yet been found in the literature.

\section{Example models} \label{Sec:ExampleModels}

In this section, we want to introduce some example models that satisfy properties \ref{V1} and \ref{V2} for every choice of $q < \alpha$.
To begin with, we consider the Boolean model.

\begin{Bsp} \label{Bsp:BooleschesModell}
Let $\AA:=\mathcal{K}^d:=\{ K\subset\R^d\mid \emptyset\neq K \text{ is convex and compact}\}$ and choose the connection functions
\begin{align*}
    \varphi_j\big((x_0,K_0),\dots,(x_j,K_j)\big) \,=\, \1\Big\{ \bigcap_{i=0}^j (x_i+K_i)\neq\emptyset \Big\}, \qquad j\in\{1,\dots,\alpha\}.
\end{align*}
Assume that the mark distribution $\Theta$ satisfies the following two properties.
\begin{enumerate}
\item[(B1)] There exists an $r_0 > 0$ such that
\begin{align*}
\Theta\big( \{ K\in\mathcal{K}^d \;\|\; B(\0,r_0)\subseteq K  \} \big) \,>\, 0.
\end{align*}
\item[(B2)] There exists an $R > 0$ such that
\begin{align*}
\Theta\big( \{ K\in\mathcal{K}^d \;\|\; K\subseteq B(\0,R) \} \big) \,=\, 1.
\end{align*}
\end{enumerate}
We demonstrate that, in this situation, the conditions \ref{V1} and \ref{V2} are satisfied for all choices of $q<\alpha$.
On the one hand, with $A := \{ K\in\mathcal{K}^d \;\|\; B(\0,r_0)\subseteq K  \}$, it holds for all $K_0, \dots, K_{q+1} \in A$ and $x_0, \dots, x_{q+1} \in \R^d$ with $\diam(\{x_0, \dots, x_{q+1}\}) \leq r_0$
\begin{align*}
\kappa_{q+1}\big( (x_0,K_0), \dots, (x_{q+1},K_{q+1}) \big) \,&=\, \1\Big\{ \bigcap_{i=0}^{q+1} (x_i + K_i) \neq \emptyset \Big\} \\
&\geq\, \1\Big\{ \bigcap_{i=0}^{q+1} B(x_i, r_0) \neq \emptyset \Big\} \,=\, 1,
\end{align*}
which shows that \ref{V1} is satisfied with $\varepsilon = 1$ and $\delta = r_0$.
On the other hand, for all $K_0, \dots, K_{q+1} \in\mathcal{K}^d$ with  $K_i \subseteq B(\0,R)$, $i\in\{0,\dots,q+1\}$ and $x_0, \dots, x_{q+1} \in \R^d$ with $\diam(\{x_0, \dots, x_{q+1}\}) > 2R$, we have
\begin{align*}
\kappa_{q+1}\big( (x_0,K_0), \dots, (x_{q+1},K_{q+1}) \big) \,&\leq\, \1\Big\{ \bigcap_{i=0}^{q+1} B(x_i, R) \neq \emptyset \Big\} \,=\, 0,
\end{align*}
which shows that \ref{V2} is satisfied with the choice $D = 2R$.
\end{Bsp}

Note that $q$-percolation of this model is equivalent to percolation of the $(q+1)$-times covered set in the Boolean model.
In \cite{Last.OSSS} a sharp phase transition for the Boolean model with the two conditions (B1),(B2) is proven.
To present another example class of connection functions that satisfy the assumptions of Theorem \ref{Th:scharferPhasenübergang}, we consider the unmarked stationary model, which arises from the marked case by choosing $\AA$ as a singleton set.
To simplify notation, we completely omit the mark space $\AA$ in the following.

\begin{Bsp} \label{Bsp:Klasse_Verbindungsfkt}
Let $\phi_1,\dots,\phi_{q+1}:\R_{\geq 0}\rightarrow [0,1]$ be monotonically decreasing functions that do not vanish almost everywhere, and let $D>0$.
Define connection functions $\varphi_j:(\R^d)^{j+1}\rightarrow [0,1]$, $j\in\{1,\dots,q+1\}$, by
\begin{align*}
\varphi_j(x_0,\dots,x_j) \,&:=\, \phi_j\big( \diam(\{x_0,\dots,x_j\}) \big), \qquad j\in\{1,\dots,q\} \\
\varphi_{q+1}(x_0,\dots,x_{q+1}) \,&:=\, \phi_{q+1}\big( \diam(\{x_0,\dots,x_{q+1}\}) \big) \; \1\big\{ \diam(\{x_0,\dots,x_{q+1}\})\leq D \big\}.
\end{align*}
Then, for each $j\in\{1,\dots,q+1\}$, there exists a $b_j>0$ with $\phi_j(b_j)>0$ (with $b_{q+1}\leq D$).
Using the notation $b:=\min_{1\leq j\leq q+1} b_j>0$ and $c:=\min_{1\leq j\leq q+1}  \phi_j(b_j)>0$, it follows for all $x_0,\dots,x_{q+1}\in\R^d$ with $\diam(\{x_0,\dots,x_{q+1}\})\leq b$ that
\begin{align*}
\kappa_{q+1}(x_0,\dots,x_{q+1}) \,&=\, \prod_{\substack{I\subseteq \{0,\dots,q+1\}, \\ |I|\geq 2}} \varphi_{|I|-1}(x_I) \\
&\geq\, \prod_{\substack{I\subseteq \{0,\dots,q+1\}, \\ |I|\geq 2}} c \,=\, c^{(q+2)^2-(q+3)} \,>\,0.
\end{align*}
Here, $\varphi_{|I|-1}(x_I)$ denotes the function value of $\varphi_{|I|-1}$ when inserting the arguments $x_i$, $i\in I$.
Thus, the connection functions $\varphi_1,\dots,\varphi_{q+1}$ satisfy property \ref{V1}.
Property \ref{V2} is obviously fulfilled by definition of $\varphi_{q+1}$.
\end{Bsp}

As a generalization of Example \ref{Bsp:Klasse_Verbindungsfkt}, condition \ref{V1} is satisfied if there exists a monotonically decreasing function $\phi:\R_{\geq 0}\rightarrow [0,1]$, which does not vanish almost everywhere, such that
\begin{align*}
    \varphi_j\big( (x_0,a_0),\dots,(x_j,a_j)\big) \,\geq\, \phi\big(\diam(\{x_0,\dots,x_{j}\})\big), \qquad j\in\{1,\dots,q+1\}.
\end{align*}
In this case, $A=\AA$, $\delta=b$, and $\varepsilon=\phi(b)^{(q+2)^2-(q+3)}$ can be chosen, where $b>0$ is any number such that $\phi(b)>0$.
Finally, we would like to highlight the Vietoris-Rips and the \v Cech complex, which also satisfy the conditions of Theorem \ref{Th:scharferPhasenübergang}.

\begin{Bsp}
The Vietoris–Rips and the \v Cech complex satisfy properties \ref{V1} and \ref{V2}.
The former corresponds to the choice of connection functions
\begin{align*}
\varphi_1(x,y) \,:=\, \1\big\{ \Vert x-y\Vert \leq 2r \big\}, \qquad \varphi_j\equiv 1, \quad j\in\{2,\dots,q+1\},
\end{align*}
for some $r>0$.
Due to the definition of the edge function, we may assume without loss of generality that $\varphi_j(x_0,\dots,x_j):=\1\{ \diam(\{x_0,\dots,x_{j}\})\leq 2r \}$ without changing the model, making it immediately clear that this is a special case of Example \ref{Bsp:Klasse_Verbindungsfkt}.
On the other hand, the \v Cech complex with parameter $r>0$ corresponds to the choice of connection functions
\begin{align*}
\varphi_j(x_0,\dots,x_j) \,:=\, \1\Big\{ \cap_{i=0}^j B(x_i,r)\neq\emptyset \Big\}, \quad j\in\{1,\dots,q+1\},
\end{align*}
which makes it a special case of Example \ref{Bsp:BooleschesModell} with $\Theta=\delta_{B(\0,r)}$.
\end{Bsp}

\section*{Acknowledgements}

The results of this paper stem from the author's PhD thesis \cite{Pabst.Thesis}.
The author is grateful to Daniel Hug, the supervisor of the PhD thesis, for the support and insightful feedback throughout the research.
This work was partially supported by the Deutsche Forschungsgemeinschaft (DFG, German Research Foundation) through the SPP 2265, under grant numbers HU 1874/5-1 and ME 1361/16-1.

\bigskip
\bigskip
\bigskip
\bigskip
\bigskip

\bibliographystyle{abbrv}
\bibliography{RandomConnectionModel_final}

\end{document}